\documentclass[11pt]{amsart}
\usepackage[final]{hyperref}
\usepackage{esint,amsmath,amssymb,amsthm}
\usepackage{mathrsfs}
\usepackage{a4wide}
\usepackage{tocvsec2}
\usepackage{enumerate}
\usepackage{wasysym}
\usepackage{multirow}
\usepackage{graphicx}
\usepackage{color}
\usepackage[normalem]{ulem}
\usepackage{pdfpages}

\newtheorem{theorem}{Theorem}[section]
\newtheorem{proposition}[theorem]{Proposition}
\newtheorem{corollary}[theorem]{Corollary}
\newtheorem{lemma}[theorem]{Lemma}
\theoremstyle{definition}

\theoremstyle{remark}

\newtheorem{remark}{Remark}[section]

\numberwithin{equation}{section}

\usepackage[abbrev]{amsrefs}
\DeclareMathOperator{\sgn}{sgn}

\newcommand{\ud}{\,\mathrm{d}}
\newcommand{\eps}{\varepsilon}

\newcommand{\R}{\mathbb R}
\newcommand{\E}{{\mathcal E}}
\newcommand{\D}{\mathcal D}
\renewcommand{\H}{{\mathcal H}}
\newcommand{\HH}{{\mathring{H}^{1/2}(\R^2)}}
\newcommand{\HHH}{{\mathring{H}^{-1/2}(\R^2)}}
\newcommand{\oH}{\mathring{H}}
\newcommand{\br}{{\bar\rho}}
\newcommand{\loc}{\mathrm{loc}}

\begin{document}

\title[Thomas-Fermi theory of charge screening in
graphene]{Thomas-Fermi theory of out-of-plane charge screening in
  graphene}

\author{Vitaly Moroz}
\address{Swansea University, Department of
    Mathematics, Fabian Way, Swansea SA1~8EN, Wales, UK}
  \email{v.moroz@swansea.ac.uk}
  
\author{Cyrill B.~Muratov}
\address{Dipartimento di Matematica, Universit\`a di Pisa, Largo
    B. Pontecorvo, 5, 56127 Pisa, Italy}
\address{Department of Mathematical Sciences, New Jersey Institute
    of Technology, University Heights, Newark, NJ~07102, USA}
\email{cyrill.muratov@unipi.it}

\date{\today}

\begin{abstract}
  This paper provides a variational treatment of the effect of
  external charges on the free charges in an infinite free-standing
  graphene sheet within the Thomas-Fermi theory. We establish
  existence, uniqueness and regularity of the energy minimizers
  corresponding to the free charge densities that screen the effect of
  an external electrostatic potential at the neutrality point. For the
  potential due to one or several off-layer point charges, we also
  prove positivity and a precise universal asymptotic decay rate for
  the screening charge density, as well as an exact charge
  cancellation by the graphene sheet. We also treat a simpler case of
  the non-zero background charge density and establish similar results
  in that case.
\end{abstract}


\maketitle
\tableofcontents

\section{Introduction}\label{sec:introduction}
\settocdepth{section}

Graphene is a classical example of a two-dimensional material whose
electronic properties give rise to a number of unusual characteristics
that make it a prime target for both fundamental research and multiple
applications
\cite{geim07,novoselov12,katsnelson,editorial14,bhimanapati15}. A key
feature of the electrons in single layer graphene sheets is the
presence of the Dirac cone in their dispersion relation that makes the
elementary excitations (electrons and holes) of the ground state
behave as massless relativistic fermions
\cite{KotovUchoa:12,castroneto09}. This presents challenges in the
theoretical treatment of those excitations, as their kinetic energy,
which is on the order of $E_K \sim \hbar v_F / r$, where
$v_F \simeq 1 \times 10^8$ cm/s is the Fermi velocity and $r$ is the
radius of the wave packet containing a single charge, remains
comparable to the Coulombic interaction energy
$E_C \sim e^2 / (\epsilon_d r)$ of two charges at distance $r$
independently of the scale $r$ (here $e$ is the elementary charge, in
the CGS units, $\epsilon_d \sim 1$ is the effective dielectric
constant, and it is noted that $e^2 / (\hbar v_F) \simeq 2.2$). As a
result, many-body effects need to be taken into consideration in the
studies of electronic properties of graphene. In particular, these
effects are significant in determining the way the massless
ultrarelativistic fermions screen the electric field of supercritical
charged impurities \cite{KotovUchoa:12}.

The problem of characterizing the charged impurity screening by the
graphene sheet has been studied, using a number of theoretical
approaches \cite{divincenzo84,shung86,shytov07,Katsnelson:06,
  HwangDasSarma:07,HainzlLewinSparber:2012,LMM} (this list is not
intended to be exhaustive). Note that a similar question arises in the
studies of the graphene based devices in the proximity of a conducting
electrode, or when a scanning tunneling microscope tip approaches a
graphene sheet \cite{alyobi20}. In particular, in this situation the
electric charge the layer is exposed to may exceed the elementary
charge $e$ by many orders of magnitude. Under such conditions, a fully
nonlinear treatment of the screening problem is, therefore, necessary.

In conventional quantum systems, a good starting point for the
analysis of electric field screening is the Thomas-Fermi theory, as it
yields an asymptotically exact response of a system of interacting
electrons to a large external charge \cite{Lieb:81}. Such a theory for
massless relativistic fermions was developed by Di Vincenzo and Mele
in the context of charged impurity screening in graphite intercalated
compounds \cite{divincenzo84}. They conducted numerical studies of the
resulting equations for the screening charge density and noted a
highly non-local character of the response. More recently, Katsnelson
carried out a formal analysis of the asymptotic behavior of the
screening charge density away from a single impurity ion in a graphene
monolayer \cite{Katsnelson:06}. His results were further clarified an
extended by Fogler, Novikov and Shklovskii, who also confirmed the
predictions about the decay of the screening charge density by
numerical simulations \cite{fogler07}. The nonlocal character of the
response and its dependence on the level of doping have been confirmed
by the direct experimental observations of the screening charge
density \cite{wang12,wang13,wong17}. Note that these observations are
at variance with the prediction of a purely local dielectric response
at the Dirac point from the linear response theory for massless
relativistic fermions within the random phase approximation
\cite{HwangDasSarma:07}.

This paper is a mathematical counterpart of the studies in
\cite{divincenzo84, Katsnelson:06, fogler07} that provides a suitable
variational framework for the study of the charge screening problem
described by the Thomas-Fermi theory of graphene (for a closely
related Thomas-Fermi-von Weizs\"acker model and some further
discussion, see \cite{LMM}). The setting turns out to be rather
delicate, as the presence of a bare Coulombic potential from an
impurity leads to heavy tails in the potential term that are precisely
balanced with the Coulombic interaction term. Within our setting, we
prove existence, uniqueness, radial symmetry and monotonicity of the
minimizer of the graphene Thomas-Fermi energy for an off-layer
external point charge in a free-standing graphene sheet. More
generally, we provide existence, uniqueness, the Euler-Lagrange
equation that is understood in a suitable sense, and regularity of the
minimizer for a general class of external potentials arising as
Coulombic potentials of appropriate collections of external
charges. Back to a single off-layer charge in a free-standing graphene
sheet, we establish the precise asymptotic decay of the screening
charge density at infinity, which agrees with the one obtained by
Katsnelson using formal arguments.

The decay of the screening charge density turns out to be a borderline
power law decay modulated by a logarithmic factor that makes it barely
integrable. The latter presents a significant technical difficulty in
the handling of the appropriate barrier functions that control the
decay of the solution at infinity. In particular, we prove that the
decay indeed tuns out to be universal, independently of the strength
of the external charge and remains the same for a finite collection of
charges of the same sign.

As a by-product of our analysis, we also demonstrate the existence of
sign-changing minimizers in the case of positive fast decaying
potentials for the closely related Thomas-Fermi-von Weizs\"acker model
studied in \cite{LMM} in the regime when the latter is well
approximated by the Thomas-Fermi model.  This gives a partial answer
to the question raised in \cite{LMM}.  Finally, we present the
corresponding results for the biased layer.  The treatment of the
latter is significantly simpler due to the expected fast power law
decay of the screening charge density.

Our paper is organized as follows. In section \ref{sec:results}, we
introduce the Thomas-Fermi energy functional for a free-standing
graphene sheet and then discuss several issues associated with its
definition in the context of the associated variational problem for
charge screening that require a modified formulation compared to the
classical Thomas-Fermi theory. Within these modifications, we then
state the main results of our paper in Theorems \ref{thm:Egeneral} and
\ref{thm-main}, and Corollary \ref{cor:cluster}.  In section
\ref{sec:setting0}, we give the precise variational setting for the
modified Thomas-Fermi energy of the free-standing graphene sheet and
establish general existence and regularity results for the minimizers.

Then, in section \ref{s:pos} we focus on the case of the potential
from a single off-layer external point charge. In particular, in
section \ref{sec:half} we reformulate the Euler-Lagrange equation for
the minimizers in terms of a convenient auxiliary variable and
establish several properties of the solutions associated with a
comparison principle that we establish for this equation, and in
section \ref{sec:super} we establish further implications of the
comparison principle on the positivity of solutions. This leads us, in
section \ref{sec:sign}, to establish existence of sign-changing
solutions to the closely related Thomas-Fermi-von Weizs\"acker model
considered by us in \cite{LMM}.

The key computation of the paper is carried out in section \ref{ALog},
where a logarithmic barrier is established, which is then used in
section \ref{sec:decay} to prove the asymptotic decay rate of the
solution at infinity for the external potential of a point
charge. Furthermore, in section \ref{sec:charge} we show the complete
charge screening and in section \ref{sec:uni} we establish the
universality of the decay. We conclude this section by showing how the
statements of our main results in section \ref{sec:results} follow
from the various technical results obtained in sections
\ref{sec:setting0} and \ref{s:pos}.

Finally, in section \ref{sec:setting} we outline the analogous
treatment of the case of a doped graphene sheet characterized by the
presence of a uniform background charge, where the main results are
contained in Theorems \ref{p-Min} and \ref{p-Minreg}.

\paragraph{\bf Notations}
Throughout the paper, for $f(t), g(t) \geq 0$ we use the asymptotic
notations as $t \to +\infty$:

\begin{itemize}
\item $f(t)\lesssim g(t)$ if there exists $C>0$ independent of $t$
  such that $f(t) \le C g(t)$ for all $t$ sufficiently large;
  \smallskip
	
\item $f(t)\sim g(t)$ if $f(t)\lesssim g(t)$ and $g(t)\lesssim f(t)$;
  \smallskip
	
\item $f(t)\simeq g(t)$ if $f(t) \sim g(t)$ and
  $\displaystyle \lim_{t\to +\infty}\frac{f(t)}{g(t)}=1$.
\end{itemize}


As usual, $B_R(x):=\{y\in\R^N:|y-x|<R\}$, $B_R:=B_R(0)$, and $C,c,c_1$
etc., denote generic positive constants.  By $C^{\alpha}(\R^2)$ we
denote the space of all locally H\"older continuous functions of order
$\alpha\in(0,1]$ on $\R^2$, and $C^{k,\alpha}(\R^2)$ denotes higher
order H\"older spaces for $k=1,2,\dots$. For an open set
$\Omega\subseteq\R^2$, by $C^\infty_c(\Omega)$ we denote the space of
all compactly supported infinitely differentiable function with the
support in $\Omega$, while $\mathcal D'(\Omega)$ is the space of
distributions on $\Omega$, i.e. the dual space of
$C^\infty_c(\Omega)$.  For a function $f\in L^1_\loc(\Omega)$, unless
specified otherwise, the inequality $f\ge 0$ in $\Omega$ is always
understood in the distributional sense, i.e., that
$\int_{\R^2}f(x)\varphi(x) dx \ge 0$ for all
$0\le \varphi\in C^\infty_c(\Omega)$. We similarly define $f \le
0$. When we want to emphasize a {\em pointwise} (in)equality, we
always write explicitly $f(x)$.

\paragraph{\bf Acknowledgements} The work of CBM was supported, in
part, by NSF via grants DMS-1614948 and DMS-1908709. This study also
received funding from the European Union -- Next Generation EU -- PRIN
2022 PNRR Project P2022WJW9H. CBM is a member of INdAM/GNAMPA. The
authors are grateful to the anonymous referee for their suggestion to
consider multipoint configurations of charges in Corollary
\ref{cor:cluster}.

\bigskip
\bigskip

\section{Model and main results} \label{sec:results}

Thomas-Fermi (TF) energy for massless relativistic fermions in a
  free-standing graphene layer in the presence of the external
electrostatic potential $V$ takes the following form, after a suitable
non-dimensionalization \cite{Katsnelson:06}:
\begin{equation}
  \label{eq:E0TF}
  \E^{TF}_{0}(\rho)=\frac{2}{3}\!\int_{\R^2}|\rho|^{3/2}\ud^2
  x-\int_{\R^2}\rho(x) V(x) \ud^2
  x+\frac{1}{4 \pi}\iint_{\R^2\times\R^2}\frac{\rho(x) 
    \rho(y)}{|x-y|}\ud^2 x\ud^2 y.
\end{equation}
Here $\rho:\R^2\to\R$ is the {\em charge density} of charge carrying
fermionic quasiparticles (electrons and holes).  The density $\rho$ is
a sign--changing function with $\rho>0$ corresponding to electrons and
$\rho<0$ to holes.  The first, {\em Thomas--Fermi term}, is an
approximation of the kinetic energy of the uniform gas of
noninteracting particles.  The exponent $3/2$ can be deduced from
scaling considerations.
The last, nonlocal {\em Coulomb term}
\begin{equation}
  \D(\rho,\rho):=\frac{1}{4 \pi}\iint_{\R^2\times\R^2}\frac{\rho(x)
    \rho(y)}{|x-y|}\ud^2 x\ud^2 y, 
\end{equation}
is the like-charged inter-particle repulsion energy which is inherited
from $\R^3$.  The middle term is the {\em potential energy} due to the
interaction with the {\em external potential} $V:\R^2\to\R$.  In the
case of a {\em single} external point charge of magnitude $Z\in\R$
located in $\R^3$ at distance $d\ge 0$ away from the graphene layer
the external potential is
\begin{equation}
  \label{VZd}
  V_{Z,d}(x):=\frac{Z}{2 \pi \sqrt{d^2+|x|^2}},
\end{equation}
but more general potentials $V(x)$ could be considered, e.g., those
involving multiple point charge configurations:
  \begin{equation}
    \label{eq:VN}
    V_N(x) := \sum_{i=1}^N V_{Z_i,d_i}(x - x_i),
  \end{equation}
  for some $Z_i \in \R$, $d_i \geq 0$ and $x_i \in \R^2$.
  Importantly, for an unscreened system of uncompensated external
  charges (i.e., when $\sum_i Z_i \not= 0$ in \eqref{eq:VN}) one has
  $V_{N}(x)\sim 1/|x|$ as $|x|\to\infty$, since the
  quasiparticle--charge interaction is according to Coulomb's law in
  $\R^3$. For a more detailed discussion of various terms in the
  energy and the non-dimensionalization, see \cite[Section
  2]{LMM}. Notice that the energy in \eqref{eq:E0TF} is invariant
    with respect to the transformation
  \begin{align}
    \label{eq:plusminus}
    \rho \to -\rho, \qquad V \to -V,
  \end{align}
  hence when dealing with the potential $V_{Z,d}$ it is sufficient to
  restrict attention to the case $Z > 0$.

Our principal goal is to prove the existence of global minimizers of
$\E^{TF}_{0}$ and establish their fundamental properties, such as
regularity and decay estimates. At first glance the
Thomas--Fermi energy $\E^{TF}_{0}$ looks similar to its classical
three-dimensional (3D) atomic counterpart
\citelist{\cite{Lieb-Simon-77}\cite{Lieb:81}\cite{Brezis-Benilan}}.
However, there are fundamental differences within the variational
framework for graphene modelling:

\begin{itemize}
\item Unlike in the classical TF-theory for atoms and molecules where
  $\rho\ge 0$, the density $\rho$ in graphene is a sign--changing
  function. As a consequence, $\D(|\rho|,|\rho|)\ge \D(\rho,\rho)$
  which means that oscillating profiles could be energetically more
  favorable.

\item All three terms in $\E^{TF}_{0}$ with $V=V_{Z,0}$ scale at
  the same rate under the charge--preserving rescaling
  $\rho_\lambda(x)=\lambda^2 \rho(\lambda x)$.  Hence
  $\E^{TF}_{0} (\rho_\lambda)=c\lambda$ when $d=0$ for some $c \in \R$.
  Physically, this is a manifestation of the non-perturbative role of
  the Coulomb interaction in graphene.  Mathematically, this reveals
  the critical tuning of the three different terms in the energy.
\item The nonlocal term $\D(\rho,\rho)$ is formally identical to the
  usual Coulomb term in $\R^3$. However, the integral kernel
  $|x-y|^{-1}$ in $\R^2$ is associated with the Green function of the
  fractional Laplacian operator $(-\Delta)^{1/2}$. As a consequence,
  the Euler--Lagrange equation for $\E^{TF}_{0}$ transforms into
  a fractional semilinear partial differential equation (PDE)
  involving $(-\Delta)^{1/2}$, instead of the usual Laplace operator
  $-\Delta$ of the classical 3D TF-theory.
\end{itemize}

\noindent
Note that the total number of electrons and holes in the graphene
sheet is neither fixed nor bounded a priori. As a consequence, unlike
in the atomic and molecular 3D models, it is unclear if the minimizers
of $\E^{TF}_{0}$ should have a finite total charge, i.e. if they
are $L^1$--functions. This implies that regular distributions should
be included as admissible densities. Indeed, even if the density
$\rho$ is a sign--changing continuous function, it is not a priori
clear if $\rho$ can be interpreted as a charge density in the sense of
potential theory (i.e., whether $d \mu = \rho\, dx$ can be associated
to a signed measure $\mu$ on $\R^2$, making the Coulomb energy
$\D(\rho,\rho)$ meaningful in the sense of the Lebesgue integration,
see \cite{LMM}*{Example 4.1} and further references therein. This
makes the analysis of the minimizers of $\E^{TF}_{0}$
mathematically challenging.

We avoid these issues by identifying the Coulomb term $\D(\rho,\rho)$
with one-half of the square of the $\mathring{H}^{-1/2}(\R^2)$ norm of
$\rho$. The energy we consider is then
\begin{equation}
\begin{aligned}
  \label{E}
  \E_0(\rho) & := \frac{2}{3}\!\int_{\R^2}|\rho|^{3/2}\ud^2 x -
  \langle \rho, V \rangle + \frac{1}{2} \left\|\rho \right\|_{\HHH}^2,
\end{aligned}
\end{equation}
where $\langle \cdot , \cdot \rangle$ is a duality pairing
between the function $V \in L^1_\mathrm{loc}(\R^2)$ and the linear
  functional generated by $\rho$, to be specified shortly.  Sometimes
we also write $\E_0^V$ to emphasize the dependence on $V$.  It is easy
to see that the definition of $\E_0$ in \eqref{E} agrees with that of
$\E^{TF}_{0}$ when $\rho\in C^\infty_c(\R^2)$ and
  $\langle \rho, V \rangle = \int_{\R^2} V \rho \ud^2 x$.

The natural domain of definition of $\E_0$ is the class
\begin{align}
  \label{eq:Am}
  \H_0:= \HHH\cap L^{3/2}(\R^2).
\end{align}
Clearly, $\H_0$ is a Banach space with the norm
$\|\cdot\|_{\H_0}=\|\cdot\|_{L^{3/2}(\R^2)}+\|\cdot\|_\HHH$.  Its dual
space $\H_0^\prime$ can be identified with the Banach space
$\HH+L^3(\R^2)$.\footnote{Recall that
  $$\HH+L^3(\R^2)=\{f\in
  L^1_{\loc}(\R^2):f=f_1+f_2,\;f_1\in \HH, f_2\in L^3(\R^2)\}$$ is a
  Banach space with the norm
  $\|f\|_{\HH+L^3(\R^2)} := \inf(\|f_1\|_{\HH}+\|f_2\|_{L^3(\R^2)})$,
  where the infimum is taken over all admissible pairs $(f_1,f_2)$.}
Therefore, one may define $\langle \cdot, \cdot \rangle$ as the
duality pairing between $\H_0'$ and $\H_0$. More precisely, for every
$\rho \in \H_0$ and every $V = V_1 + V_2$, where $V_1 \in \HH$ and
$V_2 \in L^3(\R^2)$ we may define
\begin{align}
  \label{eq:duality}
  \langle \rho, V \rangle := {}_{\HHH} \langle \rho, V_1 \rangle_\HH
  + \int_{\R^2} \rho(x) V_2(x) \ud^2 x.
\end{align}
See Section \ref{sec:setting0} for further details and precise
definitions. 

Our first result establishes the existence of a unique minimizer for
$\E_0$.

\begin{theorem}\label{thm:Egeneral}
  For every $V \in\HH + L^3(\R^2)$ there exists a unique
  minimizer $\rho_V \in \H_0$ such that
  $\E_0(\rho_V) = \inf_{\rho \in \H_0} \E_0(\rho)$.  The minimizer
  $\rho_V$ satisfies the Euler--Lagrange equation
\begin{align}\label{e-EL-int}
  \int_{\R^2} \sgn(\rho_V)|\rho_V|^{1/2}\varphi\ud^2 x - \langle
  \varphi, V \rangle  + \langle\rho_V,
  \varphi\rangle_\HHH=0,\quad\forall\varphi\in \H_0. 
	\end{align}
Furthermore, 
\begin{itemize}
\item[$i)$]
if $(-\Delta)^{1/2}V\ge 0$ then $\rho_V\ge 0$,
\smallskip
\item[$ii)$]
if $V\in \HH \cap C^{\alpha}(\R^2)$ for some $\alpha\in(0,1]$ then $\rho_V\in \H_0\cap C^{\alpha}(\R^2)$ and
$\rho_V(x) \to 0$ as $|x| \to \infty$.
\end{itemize}
\end{theorem}

Note that the statements of this theorem, including the ones about the
positivity and H\"older continuity of the minimizer apply to
$V = V_{Z,d}$ for all $Z > 0$ and $d > 0$. The positivity of the
minimizer follows from the well-known formula
  \begin{align}
    \label{eq:VZdhalf}
    (-\Delta)^{1/2} V_{Z,d}(x) = \left. -{d \over dt} V_{Z,t + d}(x)
    \right|_{t = 0} = {Z d \over 2 \pi (d^2 + |x|^2)^{3/2}} \qquad x
    \in \R^2,
  \end{align}
  that is obtained from the interpretation of the half-Laplacian in
  $\R^2$ via harmonic extension to $\R^2 \times (0, \infty)$ (see also
  the direct calculations in \cite{FV}*{p. 258 and (6.5)}). From this
  formula, it also follows that $V = V_{Z,d} \in \HH$ for all $Z > 0$
  and $d > 0$, due to the fact that
  $(-\Delta)^{1/2} V_{Z,d} \in L^{4/3}(\R^2) \subset \HHH$ in this
  case.

If, e.g., $\rho_V\in L^{4/3}(\R^2)$,
then \eqref{e-EL-int} implies that
\begin{align}
\label{eq:EL-weak0}
  \sgn(\rho_V(x))|\rho_V(x)|^{1/2}-V(x)+\frac{1}{2\pi}
  \int_{\R^2}\frac{\rho_V(y)}{|x-y|}\ud^2  
  y=0\quad\text{for a.e. $x \in\R^2$.}
\end{align}
However, \eqref{eq:EL-weak0} is not valid for a general $V\in\H_0$,
since the nonlocal term may not be well--defined as the Lebesgue
integral. Nevertheless, we show that for any $V\in\HH$ the
Euler--Lagrange equation \eqref{e-EL-int} is equivalent to the
fractional semilinear PDE
\begin{align}\label{ELu0-intr}
  (-\Delta)^{1/2} u +  |u|u = (-\Delta)^{1/2} V \quad  \text{in $\HH$},
\end{align}
and 
\begin{align}\label{e-Phiu0intr}
u_V:=\sgn(\rho_V)|\rho_V|^{1/2}\in\HH
\end{align}
is the unique solution of \eqref{ELu0-intr}.  We further show that
\eqref{ELu0-intr} satisfies suitable weak maximum and comparison
principles. This allows us to employ barrier techniques to study the
decay of the solution $u_V$. With the aid of explicit $\log$--barrier
functions constructed in Section \ref{ALog}, we establish the main
result of this work.

\begin{theorem}\label{thm-main}
  Let $Z>0$, $d>0$ and let $V_{Z,d}$ be defined in \eqref{VZd}.  Then
  the minimizer $\rho_{V_{Z,d}} \in \H_0$ of $\mathcal E_0$ with
  $V = V_{Z,d}$ is H\"older continuous, radially symmetric
  non-increasing and satisfies
\begin{equation}\label{rho-bound}
  0<\rho_{V_{Z,d}}(x)\le V_{Z,d}(x)\quad\text{for all }x\in\R^2
\end{equation}
and
\begin{equation}\label{rho-decay}
  \rho_{V_{Z,d}}(x) \simeq \frac{1}{|x|^2\log^2 |x|} \quad\text{as
  }|x|\to\infty. 
\end{equation}
In particular, $\rho_{V_{Z,d}}\in L^1(\R^2)$ and
$\|\rho_{V_{Z,d}}\|_{L^1(\R^2)}=Z$.
\end{theorem}
\noindent Notice that the asymptotic decay rate in
  \eqref{rho-decay} is {\em universal}, i.e., it does not depend on
  either the value of the charge $Z$ or $d$ for large $|x|$. Such
  ``universality of decay'' is well--known in the standard atomic
  Thomas--Fermi theory, going back to Sommerfeld
  \cite{sommerfeld1932}, cf. \cite[Section 5]{Solovej} for a
  discussion. In TF-theory for graphene a similar universality was
  observed by Katsnelson \cite{Katsnelson:06} (see also
  \cite{fogler07}).

\begin{remark}
  The order of the estimate in \eqref{rho-decay} remains valid for a
  more general class of external potentials $V$ with sufficiently fast
  decay at infinity, see Proposition \ref{P-decay}. The significance
  of the $\log$--decay becomes clear if we note that $p=2$ plays a
  role of the Serrin's critical exponent \cite{Veron-14}*{(1.7)} for
  the equation
\begin{align}\label{e-Serrin}
(-\Delta)^{1/2} u +  |u|^{p-1}u = f \quad  \text{in $\HH$},
\end{align}
with $p>1$ and (for simplicity) nonnegative $f\in C^\infty_c(\R^2)$.
If $p>2$ the linear part in \eqref{e-Serrin} dominates and solutions
must decay as the Green function of $(-\Delta)^{1/2}$,
i.e. $|x|^{-1}$.  For $3/2<p<2$ the nonlinear part in \eqref{e-Serrin}
dominates and the solutions should have ``nonlinear'' decay rate
$|x|^{-1/(p-1)}$.  In the Serrin's critical regime $p=2$ the linear
and nonlinear parts balance each other, which leads to the
$\log$--correction in the decay asymptotics, correctly captured by
Katsnelson \cite{Katsnelson:06}. Such $\log$-correction is well-known
for the local Laplacian $-\Delta$ \cite{Veron-81}*{Theorem 3.1}. We
are not aware of similar results in the fractional Laplacian case.
\end{remark}

\begin{remark}
  If $d=0$ then $V_{Z,0}(x)=Z / (2 \pi |x|) \not\in(\HH+L^{3}(\R^2))$
  and $\E_0$ with $V = V_{Z,0}$ is unbounded below, for any $Z\neq
  0$. In fact, by scaling, $V_{Z,d}(x)=d^{-1}V_{Z,1}(x/d)$ and
  $\rho_{V_{Z,d}}(x)=d^{-2}\rho_{V_{Z,1}}(x/d)$. Then
\begin{align}
\E_0^{V_{Z,d}}(\rho_{V_{Z,d}}) =
d^{-1}\E_0^{V_{Z,1}}(\rho_{V_{Z,1}})\to-\infty,
\end{align}
as $d\to 0$. Note also that by scaling,
$\|\rho_{V_{Z,d}}\|_{L^1(\R^2)}=\|\rho_{V_{Z,1}}\|_{L^1(\R^2)}$.
\end{remark}

\begin{remark}
  \label{rem:VZd}
  Observe that by \eqref{eq:VZdhalf} we have
  $(-\Delta)^{1/2} V_{Z,d} \to Z \delta_0$ in $\mathcal D'(\R^2)$ as
  $d \to 0$, so in the case $V = V_{Z,0}$ equation \eqref{ELu0-intr}
  formally becomes
  \begin{align}\label{e-Brezis}
    (-\Delta)^{1/2} u + u^2 = Z \delta_0 \quad  \text{in
    $\mathcal D'(\R^2)$}. 
  \end{align}
  Such an equation has no positive distributional solutions, see
  \cite{Veron-14}*{Theorem 4.2}.
\end{remark}

Lastly, as a corollary to Theorem \ref{thm-main} we have the following
result that is relevant to the experiments on ion cluster screening in
single graphene sheets \cite{wang13}. This result is a direct
consequence of the universal decay estimate in \eqref{rho-decay} and
the comparison principle for \eqref{ELu0-intr}.
  \begin{corollary}
    \label{cor:cluster}
    Let $N \geq 2$ be an integer, let $x_i \in \R^2$ and
    let $Z_i > 0$ and $d_i > 0$ for all $i = 1, \ldots, N$. Then the
    minimizer $\rho_{V_N} \in \mathcal H_0$ of $\mathcal E_0$ with
    $V = V_N$, where $V_N$ is given by \eqref{eq:VN}, satisfies the
    conclusions of Theorem \ref{thm-main}.
  \end{corollary}
  In physical terms, this result implies that a cluster of
  out-of-plane charges of the same sign exhibits the same universal
  decay at infinity of the induced charge density in a graphene layer
  as a single point charge and is independent of the charge
  magnitude. Therefore, surprisingly, measuring the behaviour of the
  induced charge density far from the cluster does not provide any
  information about the cluster itself.

\section{Variational setting at the neutrality point}
\label{sec:setting0}

 \subsection{Space $\mathring{H}^{1/2}(\R^2)$}
 \label{sec:func}
 
 Recall that the homogeneous Sobolev space $\mathring{H}^{1/2}(\R^2)$
 can be defined as the completion of $C^\infty_c(\R^2)$ with respect
 to the Gagliardo norm
 \begin{align}\label{Gagliardo}
 	\|u\|_{\oH^{1/2}(\R^2)}^2 :=
 	\frac{1}{4\pi}\iint_{\R^2\times\R^2}\frac{|u(x)-u(y)|^2}{|x-y|^{3}}\ud^2
 	x \ud^2 y.
 \end{align}
By the fractional Sobolev  inequality \cite[Theorem 8.4]{Lieb-Loss}, \cite[Theorem
 6.5]{palatucci12},
 \begin{align}\label{sobolev-inequality}
 	\|u\|_{\mathring{H}^{1/2}(\R^2)}^2\ge \sqrt{\pi} \,
 	\|u\|_{L^4(\R^2)}^2,\qquad\forall u\in C^\infty_c(\R^2).
 \end{align}
 In particular, the space $\mathring{H}^{1/2}(\R^2)$ is a well-defined
 space of functions and
 \begin{align}\label{sobolev-embedd}
 	\mathring{H}^{1/2}(\R^2)\subset L^4(\R^2).
 \end{align}
 The space $\mathring{H}^{1/2}(\R^2)$ is also a Hilbert space, with
 the scalar product associated to \eqref{Gagliardo} given by
 \begin{align}\label{bi-Gagliardo}
 	\langle u,v\rangle_{\oH^{1/2}(\R^2)} :=
 	\frac{1}{4\pi}\iint_{\R^2\times\R^2}\frac{(u(x)-u(y))(v(x)-v(y))}{|x-y|^{3}}\ud^2
 	x \ud^2 y.
 \end{align}
 Recall (cf.~\cite{Fukushima}) that if $u\in \mathring{H}^{1/2}(\R^2)$
 then $u^+, u^-\in \mathring{H}^{1/2}(\R^2)$ and
 $\|u^\pm\|_{\mathring{H}^{1/2}(\R^2)}\le\|u\|_{\mathring{H}^{1/2}(\R^2)}$.
 Moreover, $\langle u^+,u^-\rangle_{\oH^{1/2}(\R^2)}\le 0$.

 The dual space to $\mathring{H}^{1/2}(\R^2)$ is denoted
 $\mathring{H}^{-1/2}(\R^2)$.  According to the Riesz representation
 theorem, for every $F\in\mathring{H}^{-1/2}(\R^2)$ there exists a
 uniquely defined {\em potential} $U_F\in\mathring{H}^{1/2}(\R^2)$ such
 that
 \begin{align}\label{weak-potential}
 	\langle U_F,\varphi\rangle_{\oH^{1/2}(\R^2)}=\langle
 	F,\varphi\rangle\qquad\forall\varphi\in \mathring{H}^{1/2}(\R^2),
 \end{align}
 where $\langle F,\cdot\rangle:\mathring{H}^{1/2}(\R^2)\to\R$ denotes
 the bounded linear functional generated by $F$,  
 $\langle \cdot ,\cdot \rangle_{\HH}$ is the inner product in
 $\HH$, and $\langle \cdot ,\cdot \rangle_{\HHH}$ will be similarly
 defined as the inner product in $\HHH$.
 Moreover,
 \begin{align}\label{isometry}
 	\|U_F\|_{\oH^{1/2}(\R^2)}=\|F\|_{\mathring{H}^{-1/2}(\R^2)},
 \end{align}
 so the duality \eqref{weak-potential} is an isometry.

 The potential $U_F\in \mathring{H}^{1/2}(\R^2)$ satisfying
 \eqref{weak-potential} is interpreted as the {\em weak solution} of
 the linear equation
 \begin{align}\label{weak-Laplace}
 	(-\Delta)^{1/2}U_F= F\qquad\text{in $\R^2$},
 \end{align}
 and we recall that for functions $u\in C^\infty_c(\R^2)$, the fractional
 Laplacian $(-\Delta)^{1/2}$ can be defined as
 \begin{align}\label{half-Laplace}
 	(-\Delta)^{1/2}u(x)= \frac{1}{4\pi}\int_{\R^2}\frac{2 u(x) - u(x+y)
 		- u(x-y)}{|y|^3}\ud^2 y\qquad(x\in\R^2),
 \end{align}
cf. \cite[Proposition 3.3]{palatucci12}.

\subsection{Regular distributions in $\HHH$ and potentials}
Recall that $\rho\in \HHH\cap L^1_{\loc}(\R^2)$ means that $\rho$ is a
{\em regular distribution} in $\mathcal D'(\R^2)$, i.e.
\begin{align}\label{phi-linear}
  \langle \rho, \varphi  \rangle : = \int_{\R^2}\rho(x)\varphi(x)
  \ud^2 
  x\quad\forall\varphi\in C^\infty_c(\R^2), 
\end{align}
and $\langle \rho, \varphi \rangle$ is bounded by a multiple of
$\|\varphi\|_{\HH}$.  Then $\langle \rho, \cdot\rangle$ is understood
as the unique continuous extension of \eqref{phi-linear} to $\HH$.
Caution however is needed as not every regular distribution
$\rho\in \HHH\cap L^1_{\loc}(\R^2)$ admits an integral representation
\eqref{phi-linear} on all of $\HH$. In other words,
$\rho\in \HHH\cap L^1_{\loc}(\R^2)$ does not necessarily imply that
$\rho w\in L^1(\R^2)$ for every $w\in \HH$. Examples of this type go
back to H.~Cartan (cf. \cite{landkof}, \cite{brezis79}, or
\cite{LMM}*{Remark 5.1} for an example from $\HHH\cap C^\infty(\R^2)$
and further references).  As a consequence, the Coulomb energy term in
$\E^{TF}$ may not be defined in the sense of Lebesgue's integration
for all $\rho\in\H_0$ and should be interpreted {\em in the
  distributional sense}, i.e., in the definition of
  $\E^{TF}_{0}$ one should replace $\D(\rho, \rho)$ with
  $\|\rho\|_{\HHH}^2$. 
Recall however that every nonnegative distribution is a measure \cite[Theorem 6.22]{Lieb-Loss}.

An alternative reinterpretation of $\D(\rho,\rho)$ can be given in
terms of potentials. Given $\rho\in \HHH\cap L^1_{\loc}(\R^2)$, let $U_{\rho}\in\HH$ 
be the uniquely defined potential of $\rho$, defined as in \eqref{phi-linear}
by the Riesz's representation theorem.
If $\rho\in L^1(\R^2,(1+|x|)^{-1}\ud^2 x)$ then the potential $U_\rho$
could be identified with the Riesz potential of the function
$\rho$, so that 
\begin{align}\label{Riesz-ae}
  U_\rho(x)=\frac{1}{2\pi}\int_{\R^2}\frac{\rho(y)}{|x-y|}\ud^2 y
  \quad\text{a.e.~in $\R^2$,} 
\end{align}
(see \cite[(1.3.10)]{landkof}).
Furthermore, according to the Hardy--Littlewood--Sobolev (HLS) inequality (cf. \cite[Section 5.1, Theorem 1]{stein70}),
if $\rho\in L^s(\R^2)$ with $s\in(1,2)$ then $U_\rho\in L^t(\R^2)$ with $\frac1t=\frac1s-\frac12$, and
\begin{align}\label{HLS}
\|U_\rho\|_{L^t(\R^2)}\le C\|\rho\|_{L^s(\R^2)}. 
\end{align}

Even if \eqref{Riesz-ae} is valid, $\rho U_\rho\not\in L^1(\R^2)$ in
general.  However, if $\varphi\in \HHH\cap L^{4/3}(\R^2)$ then
$\varphi U_\rho\in L^1(\R^2)$ by the HLS
inequality and
\begin{align}\label{Riesz-dual}
  \frac{1}{2\pi}\iint_{\R^2 \times \R^2} \frac{\rho(x)\varphi(y)}{|x -
  y|}  \ud^2 x \ud^2 y=\int_{\R^2}U_\rho(x)\varphi(x)\ud^2 x=
  \langle
  U_\rho,U_\varphi\rangle_\HH=\langle\rho,\varphi\rangle_{\HHH}. 
\end{align}
In particular,
\begin{align}\label{Riesz-quad}
  \D(\rho,\rho)=\int_{\R^2}U_\rho(x)\rho(x)\ud^2 x=
  \|U_\rho\|^2_\HH=\|\rho\|^2_{\HHH}, 
\end{align}
which means that $L^{4/3}(\R^2)\subset \HHH$ and the Coulomb energy is
well--defined on $L^{4/3}(\R^2)$ in the sense of Lebesgue's
integration.

\subsection{Existence, uniqueness and regularity of the minimizers.}
Consider the unconstrained minimization problem
\begin{align}
E_0:=\inf_{\rho \in \H_0} \E_0(\rho).
\end{align}
It is easy to prove the following.

\begin{proposition}[Existence]\label{p-Min0}
  For every $V\in \HH + L^3(\R^2)$, the TF--energy $\E_0$ admits
  a unique minimizer $\rho_V \in \H_0$ such that $\E_0(\rho_V)=E_0$.
  The minimizer $\rho_V$ satisfies the Euler--Lagrange equation
	\begin{align}\label{e-EL0}
          \int_{\R^2} \sgn(\rho_V)|\rho_V|^{1/2}\varphi\ud^2 x -
          \langle \varphi, V \rangle +  \langle\rho_V,
          \varphi\rangle_\HHH=0\quad\forall\varphi\in \H_0. 
	\end{align}
\end{proposition}

\proof It is standard to conclude from $V\in \HH + L^3(\R^2)$ that
$\E_0$ is bounded below on $\H_0$, i.e., that $E_0>-\infty$.

Consider a minimizing sequence $(\rho_n)\subset\H_0$.
Clearly
\begin{align}
\sup_n \|\rho_n\|_{L^{3/2}(\R^2)} \leq C,\qquad \sup_n \|\rho_n\|_\HHH \leq C.
\end{align}
Using weak-$*$ compactness of the closed unit ball in $\HHH$, we may
extract a subsequence, still denoted by $(\rho_n)$, such that
\begin{align}
\rho_n\rightharpoonup&\;\rho_V\qquad \text{in $L^{3/2}(\R^2)$},\label{weakP20}\\
\rho_n\stackrel{\ast}{\rightharpoonup}&\;F\qquad \text{in $\HHH$},\label{weakH20}
\end{align}
for some $\rho_V\in L^{3/2}(\R^2)$ and $F\in \HHH$.
By the definition, \eqref{weakP20} and \eqref{weakH20} mean that
\begin{align}\label{Phi-linP}
	\int_{\R^2}\rho_n(x)\varphi(x) \ud^2 x\to&
	\int_{\R^2}\rho_V(x)\varphi(x) \ud^2 x\qquad\forall \varphi\in
	L^{3}(\R^2), \\
	\label{Phi-linH}
	\langle\rho_n,\varphi\rangle=\int_{\R^2}\rho_n(x)\varphi(x) \ud^2
	x\to& \;\;\langle F,\varphi\rangle\qquad\forall \varphi\in \HH. 
\end{align}
Therefore, passing to the limit we obtain
\begin{align}\label{Phi-linH-2}
	\int_{\R^2}\rho_V(x)\varphi(x) \ud^2 x=\langle
	F,\varphi\rangle\quad\forall \varphi\in L^3(\R^2)\cap \HH.
\end{align}
In particular, $\rho_V\in\HHH$ defines a regular distribution in $\mathcal{D}'(\R^2)$
and we may identify $F=\rho_V$. This implies that
\begin{equation}\label{E-vv0}
\E_0(\rho_V) \le\liminf_{n\to\infty} \E_0(\rho_n)=E_0,
\end{equation}
which follows from the weak lower semicontinuity of the
$\|\cdot\|_{L^{3/2}(\R^2)}$ and $\|\cdot\|_{\HHH}$ norms, and the
weak continuity of the linear functionals $\langle \cdot, V \rangle$
on $\H_0$.

The uniqueness of the minimizer $\rho_V\in\H_0$ is a consequence of the strict convexity of the energy $\E_0$, which is the sum of the strictly convex kinetic energy, linear external potential energy, and positive definite quadratic Coulomb energy.

The derivation of the Euler--Lagrange equation \eqref{e-EL0} is standard, we omit the details.
\qed

\begin{remark}
  As was already mentioned, if $\rho_V\in \H_0\cap L^{4/3}(\R^2)$
  then \eqref{e-EL0} can be interpreted pointwise as the integral
    equation \eqref{eq:EL-weak0}.  However, in general the
  Euler--Lagrange equation \eqref{e-EL0} for $\E_0$ should be
  interpreted as
  \begin{align}
    \label{eq:EL-weak+0}
    \sgn(\rho_V)|\rho_V(x)|^{1/2}+ U_{\rho_V}=V\quad\text{in
    $\mathcal D'(\R^2)$},
  \end{align}
  where $U_{\rho_V}\in \HH$ is the potential of $\rho_V$ defined
    via \eqref{weak-potential}.  In particular, if 
  $\rho_V\ge 0$ then $U_{\rho_V} \geq 0$ (see \cite[Theorem 3.14]{duPlessis}) which implies $V\ge 0$ and
  \begin{align}\label{eq:EL-weak+0ineq}
    0\le \rho_V\le V^2\quad\text{in $\mathcal D'(\R^2)$}.
  \end{align}
\end{remark}

\begin{remark}\label{r-biject}
  The mapping $V\mapsto \rho_V$ is a bijection between
  $\H_0^\prime=\HH+L^{3}(\R^2)$ and $\H_0$. Indeed, the uniqueness of
  the minimizer implies that $\rho_V$ is injective. Further, it is
  clear that for any $\rho\in\H_0$,
	\begin{align}
	\label{eq:V-inv}
          V:=U_{\rho}+\sgn(\rho)|\rho(x)|^{1/2}\in \HH+L^{3}(\R^2),
	\end{align}
	which means that the mapping $\rho_V$ is also surjective.  In
        particular, this shows that non--regular {\em at infinity} distributions in
        $\HHH$ could occur amongst the minimizers. Simply choose a
        regular distribution $\rho\in\H_0$ such that $\rho\varphi\not\in L^1(\R^2)$ for some $\varphi\in\HH$ (see e.g. \cite[Example 4.1.]{LMM} for an explicit example) and generate the
        corresponding potential $V$ via \eqref{eq:V-inv}. 
\end{remark}

While for a generic $V\in \HH + L^3(\R^2)$ the information
$\rho_V\in \H_0$ is optimal, under additional restrictions on the
potential $V$ the regularity of the minimizer can be improved up to
the regularity of $V$. 

\begin{lemma}[H\"older regularity]\label{lem:regularity}
  Assume that $V\in \HH \cap C^{\alpha}(\R^2)$ for some
  $\alpha\in(0,1]$.  Then the minimizer $\rho_V\in\H_0$ additionally
  satisfies $\rho_V\in \H_0\cap C^{\alpha}(\R^2)$, and
  $\rho_V(x) \to 0$ as $|x| \to \infty$. Furthermore, the potential $U_\rho$
  could be identified with the Riesz potential of $\rho$ as in \eqref{Riesz-ae} and   
  $U_{\rho_V}\in C^{1/3}(\R^2)$.
\end{lemma}

\proof
According to \eqref{eq:EL-weak+0}, the minimizer $\rho_V\in \mathcal H_0$ satisfies
\begin{equation}
  \sgn(\rho_V)|\rho_V|^{1/2}=V-U_{\rho_V}\quad\text{in $\mathcal
    D'(\R^2)$}.
\end{equation}
Since 
$\rho_V\in \mathcal H_0\subset L^{3/2}(\R^2)$,
by the HLS-inequality \eqref{HLS} with $s=3/2$ we have 
\begin{equation}
U_{\rho_V}\in L^6(\R^2),
\end{equation}
and in particular, the potential $U_\rho$
could be identified with the Riesz potential of $\rho$ as in \eqref{Riesz-ae}.

Also, by the Sobolev inequality \eqref{sobolev-inequality}, 
\begin{equation}
V\in \HH \cap C^{\alpha}(\R^2)\subset L^4(\R^2)\cap
C^{\alpha}(\R^2).
\end{equation}
This implies  
\begin{equation}
V^2\in L^2(\R^2)\cap C^{\alpha}(\R^2).
\end{equation}
In particular, both $V$ and $V^2$ are bounded and decay to zero as $|x|\to \infty$.
Note also that $U_{\rho_V}^2\in L^3(\R^2)$.
Hence, 
\begin{equation}
|\rho_V|=(V-U_{\rho_V})^2=V^2-2V U_{\rho_V}+U_{\rho_V}^2\in L^{3/2}(\R^2)\cap L^3(\R^2).
\end{equation}
Furthermore, by H\"older estimates on Riesz potentials, we conclude that $U_{\rho_V}\in C^{1/3}(\R^2)$, see \cite{LMM}*{Lemma 4.1} or \cite{duPlessis}*{Theorem 2}. Then
\begin{equation}\label{e-beta}
|\rho_V|=(V-U_{\rho_V})^2\in C^{\beta}(\R^2),
\end{equation}
where $\beta=\min\{\alpha,1/3\}$ and $\rho_V(x) \to 0$ as $|x| \to \infty$.
If $\alpha\le 1/3$ we are done. If $\alpha>1/3$ then \eqref{e-beta} implies  $U_{\rho_V}\in C^{1,1/3}(\R^2)$, see \cite{Silvestre}*{Proposition 2.8}. Therefore, $\rho_V$ has at least the same H\"older regularity as $V$.
\qed

\begin{remark}
Similarly, one can establish higher H\"older regularity of $\rho_V$ assuming higher regularity of $V$.
For instance, using \cite{Silvestre}*{Proposition 2.8} we can conclude that if $V\in C^{1,\alpha}(\R^2)$ then
$\rho_V\in C^{1,\beta}(\R^2)$, where $\beta=\min\{\alpha,1/3\}$. However, in general the H\"older regularity of $\rho_V$ can not be improved beyond the H\"older regularity of $V$.
\end{remark}

\section{Positivity and decay} \label{s:pos}

\subsection{Half-Laplacian representation, positivity and comparison}
\label{sec:half}

Let $\rho_V\in\H_0$ be the minimizer of $\E_0$. Introduce the substitution
\begin{align}\label{e-Phiu0}
u_V:=\sgn(\rho_V)|\rho_V|^{1/2}.
\end{align}
Then $\rho_V=|u_V|u_V$ and \eqref{e-EL0} transforms into
\begin{align}\label{e-EL-00}
  \int_{\R^2} u_V(x)\varphi(x)\ud^2 x - \langle \varphi, V
  \rangle + \langle 
  U_{|u_V|u_V},\varphi\rangle_{\HH}=0\quad\forall\varphi\in \H_0. 
\end{align}

\begin{proposition}[Equivalent fractional PDE]\label{pPDE0}
  Let $V\in \HH$ and $u_V$ be defined by \eqref{e-Phiu0}. Then
  $u_V\in\HH$ and is the unique solution of the problem
	\begin{align}\label{ELu0}
          (-\Delta)^{1/2} u + |u|u =  (-\Delta)^{1/2} V \quad
          \text{in $\HH$.}
	\end{align}
\end{proposition}

\proof 
Let $\psi\in C^\infty_c(\R^2)$.  Then
$(-\Delta)^{1/2} \psi\in C^\infty(\R^2)\cap L^1(\R^2)\subset L^{4/3}\cap L^1(\R^2)\subset\H_0$
\cite{Silvestre}*{Section 2.1}.  Test \eqref{e-EL-00} with
$\varphi = (-\Delta)^{1/2} \psi$ and take into
account that in view of \eqref{Riesz-dual}, 
\begin{align}\label{Riesz-inv-0}
\langle|u_V|u_V,\varphi\rangle_\HHH=&
\langle U_{|u_V|u_V},(-\Delta)^{1/2}\psi\rangle_{\HH}\\
=&\int_{\R^2}|u_V|u_V(x)\psi(x)\ud^2 x\quad\forall\psi\in C^\infty_c(\R^2).
\end{align}
Then \eqref{e-EL-00} yields
\begin{align}\label{e-EL0+}
	\int_{\R^2} u_V(-\Delta)^{1/2}\psi\ud^2 x -
	\langle (-\Delta)^{1/2}\psi, V \rangle +  \int_{\R^2}|u_V|u_V(x)\psi(x)\ud^2 x=0\quad\forall\psi\in C^\infty_c(\R^2),
\end{align}
or equivalently,
\begin{align}
\label{ELu+0}
  (-\Delta)^{1/2} u_V -  (-\Delta)^{1/2} V  + |u_V|u_V = 0 \quad  \text{in}
  \ \mathcal D'(\R^2),
\end{align}
where $(-\Delta)^{1/2} V\in \HHH$, $|u_V|u_V=\rho_V\in \HHH$.  Hence $u_V\in \HH$, and \eqref{ELu+0}
also holds weakly in $\HH$ by density.

The uniqueness for \eqref{ELu0} follows from the Comparison Principle of Lemma \ref{l-comparison0} below.
\qed

\begin{proposition}[Positivity]\label{pPDEplus}
  Let $V\in \HH$. Assume that $(-\Delta)^{1/2}V\ge 0$ in $\R^2$.
  Then $u_V\ge 0$ in $\R^2$. If, in addition $V\neq 0$ then $u_V\neq 0$.
\end{proposition}

\proof Decompose $u_V=u_V^+-u_V^-$ and recall that
$u_V^+,u_V^-\in \HH$ and $\langle u_V^+,u_V^-\rangle_\HH\le
0$. Testing \eqref{ELu0} by $u_V^-\ge 0$ and taking into
account that $u_V|u_V|u_V^-\le 0$, we obtain
\begin{align}
  0\le \langle V,u_V^-\rangle_\HH=\langle
  u_V,u_V^-\rangle_\HH+\int_{\R^2}u_V|u_V|\,u_V^-\ud^2 x 
  \le-\langle u_V^-,u_V^-\rangle_\HH\le 0.
\end{align}
We conclude that $u_V^-= 0$.

Further, if $V\neq 0$ then $u=0$ is not a solution of \eqref{ELu0} and hence $u_V\neq 0$.
\qed

\begin{lemma}[Comparison Principle]\label{l-comparison0}
  Let $V\in \HH$.  Assume that $u,v\in \mathring{H}^{1/2}(\R^2)\cap L^3(\R^2)$
  are a super and a subsolution to \eqref{ELu0} in a smooth domain
    $\Omega\subseteq\R^2$, respectively, i.e.,
	\begin{align}
	\label{ELu-comp0}
          (-\Delta)^{1/2} u + u |u| \ge  (-\Delta)^{1/2} V  \quad
          \text{in}  \ \mathcal D'(\Omega),\\ 
          (-\Delta)^{1/2} v + v |v| \le  (-\Delta)^{1/2} V  \quad
          \text{in}  \ \mathcal D'(\Omega). 
	\end{align}
	If $\R^2\setminus\Omega\neq\varnothing$, we also assume $u\ge v$ in $\R^2\setminus\bar\Omega$.  Then $u\ge v$ in
        $\R^2$. 
\end{lemma}

\proof
Subtracting one inequality from another, we obtain
\begin{align}\label{ELu20}
  (-\Delta)^{1/2} (v-u) + v|v| - u|u| \le 0  \quad\text{in}  \
  \mathcal D'(\Omega). 
\end{align}
Let $H^{1/2}_0(\Omega)$ denotes the completion of $C^\infty_c(\Omega)$ wrt the Gagliardo's norm $\|\cdot\|_{\oH^{1/2}(\R^2)}^2$, defined in \eqref{Gagliardo}. With this definition, $H^{1/2}_0(\Omega)$
is automatically a closed subspace of $H^{1/2}_0(\R^2)$.
By density, \eqref{ELu20} is also valid in $H^{1/2}_0(\Omega)$, in the sense that
\begin{equation}\label{ELuH0}
  \langle v-u,\varphi\rangle_\HH+ \int_{\R^2}(v|v| - u|u|)\varphi\ud^2
  x\le 0\quad\forall\,0\leq \varphi\in H^{1/2}_0(\Omega).
\end{equation}
Note that $(v-u)^+\in \HH$. If
$\R^2\setminus\Omega\neq\varnothing$ then $u\ge v$ in
$\R^2\setminus\bar\Omega$ and hence $(v-u)^+=0$ in
$\R^2\setminus\bar\Omega$. This implies $(v-u)^+\in H^{1/2}_0(\Omega)$,
see e.g.~\cite{AdamsHedberg}*{Theorem 10.1.1}. Testing \eqref{ELuH0}
by $(v-u)^+$, taking into account
$\langle(v-u)^-,(v-u)^+\rangle_\HH\le 0$ and monotone increase of the
nonlinearity, we obtain
\begin{multline}
  0\ge\langle v-u,(v-u)^+\rangle_\HH+\int_{\R^2}(v|v| - u|u|)(v-u)^+\ud^2 x\\
  \ge \langle(v-u)^+,(v-u)^+\rangle_\HH=\|(v-u)^+\|_{ \mathring{H}^{1/2}(\R^2)}^2.
\end{multline}
We conclude that $(v-u)^+= 0$. 
\qed

The Comparison Principle immediately implies that \eqref{ELu0} can
have at most one solution in $\HH$.  Hence the solution $u_V$
constructed from the minimizer $\rho_V$ via \eqref{e-Phiu0} is the
unique solution of \eqref{ELu0}. A consequence of the uniqueness is
the following.

\begin{corollary}\label{c-monotone}
  Assume that $V\in \HH$ and $(-\Delta)^{1/2} V\ge 0$ in $\R^2$.  If $(-\Delta)^{1/2} V\in L^{4/3}(\R^2)$ is a radially symmetric non-increasing function then $u_V$ is also radially symmetric and
  non-increasing.
\end{corollary}

\proof Note that $u_V$ is the unique global minimizer of the convex energy
$$J_V(u)=\frac12\|u\|_\HH^2+\frac13\|u\|_{L^3(\R^2)}^3-\langle u,V\rangle_\HH$$ on $\HH\cap L^3(\R^2)$.
Since $(-\Delta)^{1/2} V\in L^{4/3}(\R^2)$, 
$$\langle u_V, V\rangle_\HH=\int_{\R^2}u_V (-\Delta)^{1/2} V \,\ud^2 x,$$
where the latter integral is finite by the HLS inequality.
Then the symmetric--decreasing rearrangement $u_V^*$ is also a minimizer of
$J_V$, by \cite{Lieb-Loss}*{Theorem 3.4 and Lemma 7.17}.  
Hence the assertion follows from the uniqueness of the minimizer.  
\qed

Another straightforward, but important consequence of the Comparison
Principle is the following upper bound on $u_V$.

\begin{corollary}\label{c-KO}
Assume that $V\in \HH$ and $V\ge 0$. Then
\begin{align}\label{e-global}
u_V\le V\quad\text{in $\R^2$.}
\end{align}
\end{corollary}

\proof
We simply note that $V$ is a supersolution to
\eqref{ELu0} in $\R^2$, i.e.
\begin{equation}
\label{ELu-comp0-V}
(-\Delta)^{1/2} V + V^2 \ge  (-\Delta)^{1/2} V  \quad  \text{in}  \ \mathcal D'(\R^2).
\end{equation}
Hence, \eqref{e-global} follows from the Comparison Principle in $\R^2$.
\qed

The Comparison Principle can be used as an alternative tool to prove
the existence of the solution $u_V$ of \eqref{ELu0}, via construction
of appropriate sub and supersolutions. In the next section we
construct an explicit barrier which later will be used to obtain lower
and upper solution with matching sharp asymptotics at infinity. This
will lead to the sharp decay estimates on $u_V$ and $\rho_V$.

\subsection{Super-harmonicity of the potential is essential}
\label{sec:super}

We are going to show that the assumptions $(-\Delta)^{1/2} V\ge 0$ is in a certain sense necessary for the positivity of the minimizer $\rho_V$.

\begin{proposition}\label{P-dipole}
	Let $V\in\HH\cap C^{\alpha}(\R^2)$ for some $\alpha\in(0,1]$.
	Assume that $V\neq 0$ and
	\begin{equation}\label{e-Z0}
	\lim_{|x|\to\infty}|x|V(x)=0.
	\end{equation}
	Then $\rho_V$ changes sign in $\R^2$.
\end{proposition}

\begin{remark}
  The assumption \eqref{e-Z0} implicitly necessitates that
  $(-\Delta)^{1/2}V$ can not be non-negative.  Indeed, if
  $(-\Delta)^{1/2}V\ge 0$ then $\lim_{|x|\to\infty}|x|V(x)>0$ (cf. \eqref{Riesz-ae-loc} below), which is incompatible with \eqref{e-Z0}.
\end{remark}

\proof According to \eqref{eq:EL-weak+0} and Lemma
  \ref{lem:regularity}, we know that $\rho_V\in \mathcal H_0\cap C^{\alpha}(\R^2)$, 
  $U_\rho$   could be identified with the Riesz potential of the function
  $\rho$ as in \eqref{Riesz-ae}, $U_{\rho_V}\in C^{1/3}(\R^2)$,
   and 
\begin{equation}
  \mathrm{sign}(\rho_V)|\rho_V|^{1/2}(x)=V(x)-
  U_{\rho_V}(x)\quad\text{for all $x\in\R^2$}.
\end{equation}
Assume that $\rho_V\ge 0$ in $\R^2$. 
Then for each $x\in\R^2$,
\begin{align}\label{Riesz-ae-loc}
	U_\rho(x)\ge\frac{1}{2\pi}\int_{B_{2|x|}(x)}\frac{\rho(y)}{|x-y|}\ud^2 y\ge \frac{1}{4\pi|x|}\int_{B_{2|x|}(x)}\rho(y)\ud^2 y.
\end{align}
In particular,
\begin{equation}
\liminf_{|x|\to\infty}|x|U_{\rho_V}(x)>0
\end{equation}
and hence, in view of \eqref{e-Z0},
\begin{equation}
  \limsup_{|x|\to\infty}|x|\mathrm{sign}(\rho_V)|\rho_V|^{1/2}(x)=
  \limsup_{|x|\to\infty}|x|(V(x)- 
  U_{\rho_V}(x))<0,
\end{equation}
a contradiction. A symmetric argument shows that $\rho_V\le 0$ is also
impossible.  \qed

\begin{remark}
	For example, we can consider the dipole potential
	$$W_Z(x)=\frac{Z}{2 \pi(1+|x|^2)^{3/2}}.$$
	Note that $W_Z(x)=-\left. {d \over dt} V_{Z,t}(x)
        \right|_{t=1}$.  While $W_Z>0$, it is not difficult to see,
        using the harmonic extension of $W_Z$, that
	$$(-\Delta)^{1/2}W_Z(|x|)=\frac{Z(2-|x|^2)}{2
            \pi(1+|x|^2)^{5/2}},$$ 
	which is a sign--changing function.
	Clearly, $W_Z$ satisfies the assumptions of Proposition \ref{P-dipole}, so the minimizer $\rho_{W_Z}$ changes sign for any $Z>0$.
\end{remark}

\subsection{Sign-changing minimizer in TFW model}
\label{sec:sign}

A density functional theory of Thomas-Fermi-Dirac-von~Weizs\"acker (TFW) type to describe the response of a single layer of graphene to a charge $V$ was developed in \cite{LMM}. For $\eps > 0$, and in the notations of the present paper, the TFW-energy studied in \cite{LMM} has the form:
\begin{equation}\label{e-dipole}
\E_{0,\eps}(\rho) :=
\eps\||\rho|^{-1/2}\rho\|_{\mathring{H}^{1/2}(\R^2)}^2+
\E_0(\rho):\H_0\to\R\cup\{+\infty\}.
\end{equation}
The existence of a minimizer for  $\E_{0,\eps}$ with $V\in \HH$ was established in \cite[Theorem 3.1]{LMM}.
We are going to show that if $V\ge 0$ satisfies the assumptions of Proposition \ref{P-dipole} then for sufficiently small $\eps>0$ the TFW--energy $\E_{0,\eps}$ {\em admits a sign--changing minimizer}. This gives a partial answer to one of the questions left open in \cite{LMM} (see discussions in \cite[Section 3]{LMM}).

To show the existence of a sign--changing minimizer for $\E_{0,\eps}$, 
assume that $V\ge 0$ and the assumptions of Proposition \ref{P-dipole} holds. Then the minimizer $\rho_V$ of $\E_0$ changes sign.
Let
$$E_0:=\inf_{\H_0}\E_0=\E_0(\rho_V).$$
Similarly to Proposition \ref{p-Min0}, we can also minimize convex energy $\E_0$ on the weakly closed set $\H_0^+$ of nonnegative functions in $\H_0$. 
Let $\rho_V^+\in\H_0^+$ be the minimizer of $\E_0$ on $\H_0^+$ and set
$$E_0^+:=\inf_{\H_0^+}\E_0=\E_0(\rho_V^+).$$
It is clear that $E_0^+<0$ and hence $\rho_V^+\neq 0$ (just take trial functions $0\le \varphi \in\mathcal D'(\R^2)$ such that $\langle V,\varphi\rangle>0$). By an adaptation of arguments in \cite[Theorem 11.13]{Lieb-Loss}, the minimizer $\rho_V^+$ satisfies the Thomas--Fermi equation
\begin{equation}\label{TF0++}
	(\rho_V^+)^{3/2}=\big(V-U_{\rho_V^+}\big)^+\quad\text{in $\mathcal D'(\R^2)$}. 
\end{equation}
Observe that  $\mathrm{supp}(\rho_V^+)\neq\R^2$. Indeed, assume that $\rho_V^+>0$ in $\R^2$.
Then $\rho_V^+>0$ satisfies the Euler-Lagrange equation
\begin{equation}\label{TF++}
	(\rho_V^+)^{3/2}=V-U_{\rho_V^+}\quad\text{in $\mathcal
		D'(\R^2)$}, 
\end{equation}
which contradicts to the uniqueness, since \eqref{TF++} has a sign--changing solution $\rho_V$ by Proposition \ref{P-dipole}. 
Crucially, by the strict convexity of $\E_0$ we can also conclude that
\begin{equation}\label{EE+}
E_0<E_0^+.
\end{equation}
Next, for $\eps > 0$ consider the TFW-energy $\E_{0,\eps}$. Set
$$E_{0,\eps}:=\inf_{\H_0}\E_{0,\eps}.$$
The existence of a minimizer for  $E_{0,\eps}$ was established in \cite[Theorem 3.1]{LMM}.
Without loss of a generality, we may assume that $\rho_V$ is regular
enough and $|\rho_V|^{-1/2}\rho\in \mathring{H}^{1/2}(\R^2)$ (otherwise we may approximate $\rho_V$ by smooths functions). 
Then
$$E_{0,\eps}\le\eps \||\rho_V|^{-1/2}\rho_V\|_{\mathring{H}^{1/2}(\R^2)}^2 
+E_0\to E_0\quad\text{as $\eps\to 0$}.$$ 
Similarly, 
$$E_{0}^+\le E_{0,\eps}^+:=\inf_{\H_0^+}\E_{0,\eps}.$$
Taking into account the strict inequality \eqref{EE+}, for sufficiently small $\eps>0$ we have
$$E_0<E_{0,\eps}< E_{0}^+\le E_{0,\eps}^+.$$
In particular, $E_{0,\eps}< E_{0,\eps}^+$ and we conclude that a minimizer for $E_{0,\eps}$ must change sign. For example, a dipole, or any compactly supported nonnegative potential should give rise to a sign--changing
global minimizer in the TFW model.

\subsection{Logarithmic barrier}\label{ALog}

Recall (cf. \cite{FV}*{Theorem 1.1}) that for a radial function
$u\in C^2(\overline{\R_+})$ such that
\begin{equation}
  \int_0^\infty \frac{|u(r)|}{(1+r)^{3}}\, r \ud r <\infty,
\end{equation}
the following representation of the fractional Laplacian
$(-\Delta)^{1/2}$ in $\R^2$ is valid:
\begin{equation}\label{VF}
  (-\Delta)^{1/2}u(r)=\frac{1}{2 \pi 
    r}\int_1^\infty\left(u(r)-u(r\tau)
    +\frac{u(r)-u(r/\tau)}{\tau}\right)\mathcal K(\tau) \ud \tau, 
\end{equation}
where
\begin{equation}
  \mathcal K(\tau) := 2\pi\tau^{-2}\,
  {}_{2}F_1(\tfrac{3}{2},\tfrac{3}{2},1,\tau^{-2}) ,
\end{equation}
see \cite{FV}*{p. 246}. Note that $\mathcal K(\tau) > 0$ and
  \begin{align}
    \mathcal K(\tau) & \sim 
                       (\tau-1)^{-2} \qquad \text{as} \qquad \tau \to 1^+,  \label{K-1}\\
    \mathcal K(\tau) & \sim
                       \tau^{-2}\qquad \text{as} \qquad \tau \to +\infty,\label{K-inf}
  \end{align}
  so the kernel $\mathcal K(\tau)$ is integrable as
$\tau\to +\infty$, but it is {\em singular} as $\tau\to 1^{+}$.

Denote
\begin{equation}
  \Phi_u(r,\tau):=u(r)-u(r\tau) +\frac{u(r)-u(r/\tau)}{\tau}.
\end{equation}
Clearly, $\Phi_u(r,1)=0$. A direct computation shows that
\begin{equation}
  \partial_\tau\Phi_u(r,1)=0,\quad\partial^2_\tau\Phi_u(r,1)=-2r^2\mathscr
  L u(r), 
\end{equation}
where the differential expression
\begin{equation}
\mathscr L u(r):=u''(r)+\frac{2}{r}u'(r)
\end{equation}
acts on $u(r)$ as the radial Laplacian in 3D.  In particular, the
integral in \eqref{VF} converges as $\tau\to 1^{+}$.  


We now define a barrier function $U \in C^2(\overline{\R_+})$ such
  that $U(r)$ is monotone decreasing and
\begin{equation}\label{e-Barrier}
  U(r) =
  \frac{1}{r\log(er)} \qquad \forall r>1.
\end{equation}
Clearly, if $u(x) := U(|x|)$ then $u\in H^{1}(\R^2)$.  By
interpolation between $L^2(\R^2)$ and $H^{1}(\R^2)$
(cf. \cite{Bahouri}*{Proposition 1.52}) we also conclude that
$u\in H^{1/2}(\R^2)$.

\begin{lemma}\label{l-Barrier} 
  There exists $R>2$ such that
	\begin{equation}\label{e-log-est}
          (-\Delta)^{1/2}U(r) \sim -
          \frac{1}{r^2(\log(r))^2}\quad\text{for all }\, r>R. 
	\end{equation}
\end{lemma}

\proof Our strategy is to split the representation in \eqref{VF} into
three parts $\int_1^2+\int_2^r+\int_r^\infty$ and then either estimate
each part from above and below or compute the integrals
explicitly, see \eqref{Lsub} and \eqref{Lsuper}.

For $r>2$ we compute
\begin{equation}
\mathscr L U(r)=\frac{\log(e^3 r)}{(r\log(er))^3}>0.
\end{equation}
Next we claim that for all $r>2$ the following inequalities hold:
\begin{align}
  \Phi_U(r,\tau)&< U(r)\quad&\forall&\tau\in[r,+\infty),\label{i1}\\
  \Phi_U(r,\tau)&\le 0 \quad &\forall&\tau\in[1,r],\label{i2}\\
  \Phi_U(r,\tau)&\ge -4 r^2\mathscr{L}
                  U(r)(\tau-1)^2\quad&\forall&\tau\in[1,2].\label{i3} 
\end{align}

We begin by noting that by monotonicity and positivity of $U$ we
have
\begin{equation}
  \Phi_U(r,\tau)
  < U(r),
\end{equation}
which yields \eqref{i1}. To deduce \eqref{i2}, observe that for
$r>2$ and $1\le\tau\le r$ we have
\begin{equation}
  \Phi_U(r,\tau)=\frac{1}{r}\left\{\frac{1}{\log(er)}-
    \frac{1}{\log(er/\tau)}+
    \frac{1}{\tau}\left(\frac{1}{\log(er)}-\frac{1}{\log(er\tau)}
    \right)\right\}.    
\end{equation}
It is elementary to see that \eqref{i2} is equivalent to
\begin{equation}
\frac{\log(er\tau)}{\log(er/\tau)}\ge\frac{1}{\tau},
\end{equation}
the latter is true for any $r>1$ and $\tau\in[1,r]$ (since in this
range the left hand side is bigger than one).

To derive \eqref{i3}, let $A:=\log(er)$ and observe that for $r>2$ and
$\tau\in[1,2]$ we have $A>1$ and
\begin{multline}\label{i3-1}
  r\,\left\{\Phi_U(r,\tau)+4\mathscr{L} U(r)r^{2}(\tau-1)^2\right\}=\\
  =\frac{1}{\log(er)}-\frac{1}{\log(er)-\log(\tau)}+\frac{1}{\tau}
  \left(\frac{1}{\log(er)}-\frac{1}{\log(er)+\log(\tau)}\right)
  +\frac{4\log(e^3 r)}{(\log(er))^3}(\tau-1)^2\\
  =\frac{1}{A}\left(1+\frac{1}{\tau}\right)-
  \left(\frac{1}{A-\log(\tau)}+\frac{1}{\tau(A+\log(\tau))}\right)
  +\frac{4(2+A)}{A^3} (\tau-1)^2\\
  \ge\frac{1}{A}\left(1+\frac{1}{\tau}\right)-
  \left(\frac{1}{A-\log(\tau)}+\frac{1}{\tau(A+\log(\tau))}\right)
  +\frac{4}{A^2}(\log(\tau))^2,
\end{multline}
where we used the fact that $\log(\tau)<\tau-1$ for $\tau\ge 1$.  It
is convenient to substitute $\tau=e^x$, where $x\in[0,\log(2)]$.
Then, taking into account that $A\ge\log(2e) > 2x$ we rewrite the
right-hand side of \eqref{i3-1} as
\begin{multline}
  \frac{1}{A}-\frac{1}{A-x}+e^{-x}\left(\frac{1}{A}-\frac{1}{A+x}\right)
  +\frac{4 x^2}{A^2}  = {x \over A} \left\{ -{1 \over A - x}
      + {e^{-x} \over A + x} +
      {4x \over A}  \right\} \\
  \ge\frac{x}{A^2}\left\{-1-\frac{2x}{A}+
      (1-x)\left(1-\frac{x}{A} \right)
      +4 x\right\} \\
  \ge\frac{3
      x^2}{A^2}\left\{1-\frac{1}{A}\right\}
  \ge 0\quad\text{for all $x\in[0,\log(2)]$}.
\end{multline}


Now, for $r>2$, we compute explicitly, using again the
  substitution $\tau = e^x$ and a standard asymptotic expansion of the
  integral:
\begin{multline}\label{Wolf-2}
  \int_2^r r \Phi_U(r,\tau)\tau^{-2}d\tau=
  \int_{\log(2)}^{z^{-1}} \left(\frac{x z^2 e^{-x} }{(z+1) (x
        z+z+1)}+\frac{z}{(x-1)
        z-1}+\frac{z}{z+1}\right) e^{-x}  dx \\
= -\frac{7+6\log(2)}{16} z^2 + O(z^3) \qquad \text{as} \qquad z
  \to 0^+,
\end{multline}
where we defined $z := 1/\log(r)$. Similarly, we have
\begin{multline}
  \label{Wolf-R}
  \left| \int_r^\infty r \Phi_U(r,\tau)\tau^{-2}d\tau \right| \leq
    \int_{z^{-1}}^\infty \frac{z e^{-x}}{z+1} dx + U(0)
    e^{z^{-1}}\int_{z^{-1}}^\infty e^{-2x} dx \leq (1 + \tfrac12 U(0))
    e^{-z^{-1}}.
\end{multline}
Therefore, taking into account \eqref{K-inf} and using \eqref{i1},
\eqref{i2} and \eqref{Wolf-2}, for $r>2$ we estimate
\begin{multline}\label{Lsub}
  (-\Delta)^{1/2}U(r)\lesssim
  r^{-1}\int_2^r\Phi_U(r,\tau)\tau^{-2}d\tau
  +r^{-1}U(r)\int_{r}^\infty\mathcal \tau^{-2}d\tau,\\
  \sim-\frac{1}{r^{2}(\log(r))^2}+\frac{1}{r^{3}\log(r)}
  \sim - \frac{1}{r^{2}(\log(r))^2} \qquad \text{as} \qquad r
    \to \infty.
\end{multline}
To deduce a lower estimate, we use \eqref{i3}, \eqref{Wolf-R} and
\eqref{Wolf-2} to obtain
\begin{multline}\label{Lsuper}
  (-\Delta)^{1/2}U(r)\gtrsim-r\mathscr{L}U(r)+
  r^{-1}\left(\int_{2}^r+\int_r^\infty\right)\Phi_U(r,\tau)
  \tau^{-2} d\tau,\\
  \gtrsim-\frac{1}{r^{2}(\log(r))^2}-\frac{1}{r^{2}(\log(r))^2}-
  \frac{1}{r^{3}} \sim -\frac{1}{r^{2}(\log(r))^2}
  \qquad \text{as} \qquad r \to \infty,
\end{multline}
which completes the proof.
\qed

\subsection{Decay estimate}
\label{sec:decay}

\begin{proposition}\label{P-decay}
	Let $V\in\HH\cap C^{\alpha}(\R^2)$ for some $\alpha\in(0,1]$.
	Assume that $(-\Delta)^{1/2}V\geq 0$, $V\neq 0$, and for some $R>0$ and $C>0$,
	\begin{equation}\label{V-Logdec}
	(-\Delta)^{1/2}V\le\frac{C}{|x|^2(\log|x|)^2}\quad\text{for }|x|\ge R.
	\end{equation}
	Then the unique solution $u_V\in H^{1/2}(\R^2)\cap C^{\alpha}(\R^2)$ of \eqref{ELu0} satisfies 
	\begin{equation}\label{eV-upper}
	0<u_V(x)\le V(x)\quad\text{for all }x\in\R^2
	\end{equation}
	and
	\begin{equation}\label{eV-upperLog}
	u_V(x)\sim\frac{1}{|x|\log|x|}\quad\text{as }|x|\to\infty.
	\end{equation}
	In particular, $u_V\in L^2(\R^2)$. 
\end{proposition}

\begin{remark}
  \label{rem:pos}
  We do not assume radial symmetry of $V$ or $u_V$.  The assumptions
  $(-\Delta)^{1/2}V\ge 0$ and $V\neq 0$ ensure the positivity of
  $u_V$, while the upper bound \eqref{V-Logdec} controls the
  logarithmic decay rate \eqref{eV-upperLog}.  The bound
  \eqref{V-Logdec} together with $(-\Delta)^{1/2}V\ge 0$ implicitly
  necessitates that $V$ is positive in $\R^2$,
  $(-\Delta)^{1/2}V\in L^1(\R^2)$ and
	\begin{equation}\label{eV+}
          \lim_{|x|\to\infty}2 \pi |x| V(x)=\|(-\Delta)^{1/2}V\|_{L^1(\R^2)},
	\end{equation}
	see Lemma \ref{lemma-Newton} below.  	
\end{remark}

\proof Note that $(-\Delta)^{1/2}V\ge 0$ implies that $V\ge 0$ (this
could be seen similarly to the argument in the proof of Proposition
\ref{pPDEplus} but without the nonlinear term).  Then the upper bound
in \eqref{eV-upper} follows by Corollary \ref{c-KO}.  Next recall that
$u_V\in C^{\alpha}(\R^2)$ by Lemma \ref{lem:regularity} and
$u_V\neq 0$ by Proposition \ref{pPDEplus}.  Therefore, with
$c := \|u_V\|_{L^\infty{(\R^2)}}$ we get
$$((-\Delta)^{1/2}+c)u_V=(c-u_V)u_V+(-\Delta)^{1/2}V\ge 0\quad\text{in $\R^2$}.$$
This implies that $u_V(x)>0$ for all $x\in\R^2$, cf. \cite{LMM}*{Lemma 7.1}.

To derive \eqref{eV-upperLog}, set $U_\lambda:=\lambda U$, where
$U$ is the logarithmic barrier function defined in \eqref{e-Barrier}.
Recall that $U\in H^{1/2}(\R^2)\subset\HH$. Using \eqref{e-log-est} to
estimate $(-\Delta)^{1/2}U_\lambda$, we conclude that there exist
  positive constants $c_1, c_2, C$ such that for some $R'>R$ and all
sufficiently large $\lambda>0$,
\begin{multline}\label{B-super}
  (-\Delta)^{1/2}U_\lambda+bU_\lambda^2-(-\Delta)^{1/2}V\ge\\
  \ge-\frac{c_1\lambda}{|x|^{2}(\log(|x|))^2}+\frac{\lambda^2}{|x|^{2}
    (\log(e|x|))^2}-\frac{C}{|x|^2(\log|x|)^2}\ge 0 \quad\text{for
  }|x|\ge R'.
\end{multline}
Similarly, for some $R'>R$ and all sufficiently small $\lambda>0$,
\begin{equation}\label{B-sub}
(-\Delta)^{1/2}U_\lambda+bU_\lambda^2-(-\Delta)^{1/2}V\le
-\frac{c_2\lambda}{|x|^{2}(\log(|x|))^2}+\frac{\lambda^2}{|x|^{2}(\log(e|x|))^2}\le 0
\quad\text{for }|x|\ge R'.
\end{equation}
Therefore, for suitable values of $\lambda$ we can use $U_\lambda$ as a sub or supersolution in the Comparison Principle of Lemma \ref{l-comparison0} with $\Omega=B_R^c$.

To construct a lower barrier for the solution $u_V$, set $\lambda_0:=\min_{\bar B_R}u_V>0$.
Then
\begin{equation}
u_V\ge U_{\lambda_0}\quad\text{in }\,\bar B_R.
\end{equation}
Taking into account \eqref{B-sub}, we conclude by Lemma \ref{l-comparison0} that
\begin{equation}
u_V\ge U_\lambda\quad\text{in }\,\R^2,
\end{equation}
for a sufficiently small $\lambda\le\lambda_0$.

To construct an upper barrier for $u_V$, choose $\mu>0$ such that
\begin{equation}
u_V\le U_\mu\quad\text{in }\,\bar B_R,
\end{equation}
Using \eqref{B-super}, we conclude by Lemma \ref{l-comparison0} that
\begin{equation}
u_V\le U_\lambda\quad\text{in }\,\R^2,
\end{equation}
for a sufficiently large $\lambda\ge\mu$.
\qed

\subsection{Charge estimate}
\label{sec:charge}

In the case of the standard Newtonian kernel $|x|^{-1}$ on $\R^3$ it
is well--known that for a nonnegative $f\in L^1_{rad}(\R^3)$,
$|x|^{-1}*f=\|f\|_{L^1(\R^3)}|x|^{-1}+o(|x|^{-1})$ as
$|x|\to\infty$, cf. \cite{Siegel} for a discussion. The result becomes nontrivial when we
consider the convolution kernel $|x|^{-1}$ on $\R^2$, or more
generally the Riesz kernel $|x|^{-(N-\alpha)}$ on $\R^N$ with
$\alpha\in(0,N)$.  It is known that if $\alpha\in(1,N)$ and
$f\in L^1(\R^N)$ is positive radially symmetric then
$|x|^{-(N-\alpha)}*f=O(|x|^{-(N-\alpha)})$, see \cite{Siegel}*{Theorem
	5(i)}.  The same remains valid if $\alpha\in(0,1]$ and $f$ is in
addition monotone decreasing, see \cite{Duoandi}*{Lemma 2.2
	(4)}. However, without assuming monotonicity of $f$,
$|x|^{-(N-\alpha)}*f$ with $\alpha\in(0,1]$ could have arbitrary fast
growth at infinity \cite{Siegel}*{Theorem 5}.

We are going to show that if $f$ is monotone non-increasing and decays
faster than $|x|^{-2}$ then the sharp asymptotics of $|x|^{-1}*f$ on
$\R^2$ is recovered. The proof is easily extended to Riesz kernels
with $N\ge 2$ and $\alpha\in(0,N)$.

\begin{lemma}[Asymptotic Newton's type theorem]\label{lemma-Newton}
	Let $0\le f\in L^1(\R^2)$ be a function dominated by a radially
	symmetric non-increasing function $\varphi:\R_+\to\R_+$ that
	satisfies
	\begin{equation}\label{logass}
	\lim_{|x|\to\infty}\varphi(x)|x|^2=0.
	\end{equation}
	Then 
	\begin{equation}\label{e-L1}
	\int_{\R^2}\frac{f(y)}{|x-y|}\,\ud^2 y=
	\frac{\|f\|_{L^1(\R^2)}}{|x|}+o(|x|^{-1})\quad\text{as    
		$|x|\to\infty$}. 
	\end{equation}
\end{lemma}

\begin{proof}
	Fix $0\neq x\in\R^2$ and decompose $\R^2$ as the union of
	$B=\{y: |y-x|<|x|/2\}$, $A=\{y\not\in B: |y|\le |x|\}$,
	$C=\{y\not\in B: |y|>|x|\}$.
	
	We want to estimate the quantity
	\begin{equation}
	\left|\int_{A\cup C} f (y) \Bigl(\frac{1}{|x - y|} -
	\frac{1}{|x|} \Bigr) \, \ud^2 y \right|
	\le  \int_{A\cup C} f (y) \left|\frac{1}{|x - y|} -
	\frac{1}{|x|}\right| \, \ud^2 y .
	\end{equation}
	Since $|x|/2\le|x-y|\le 2|x|$ for all $y\in A$, by the Mean Value Theorem we have
	\begin{equation}
	\left|\frac{1}{|x - y|} - \frac{1}{|x|}\right|
	\le \frac{4|y|}{|x|^2}\qquad(y\in A).
	\end{equation}
	Thus
	\begin{equation}
	\label{eqRieszAsymptoticsSmall}
	\begin{split}
	\left|\int_A f (y) \Bigl(\frac{1}{|x - y|} - \frac{1}{|x|}
	\Bigr) \, \ud^2 y \right|
	& \le
	\frac{4}{|x|^2} \int_{A} f(y)|y|\, \ud^2 y.
	\end{split}
	\end{equation}
	On the other hand, since $|x-y|>|x|/2$ for all $y\in C$ then
	\begin{equation}
	\left|\frac{1}{|x|}-\frac{1}{|x - y|}\right|
	\le \frac{1}{|x|}\qquad(y\in C),
	\end{equation}
	from which we compute that
	\begin{equation}
	\label{eqRieszAsymptoticsLarge}
	\left|\int_Cf(y) \Bigl(\frac{1}{|x - y|} - \frac{1}{|x|}
	\Bigr) \, \ud^2 y \right|
	\le
	\frac{1}{|x|}\int_C f(y) \, \ud^2 y.
	\end{equation}
	Then
	\begin{multline}
	\left|\int_{\R^2}\frac{f(y)}{|x-y|}\, \ud^2 y
	-\frac{\|f\|_{L^1(\R^2)}}{|x|}\right|\le\\ 
	\frac{4}{|x|^2} \int_{A} f(y)|y|\, \ud^2
	y+\int_{B}\frac{f(y)}{|x-y|}\, \ud^2
	y+\frac{1}{|x|}\int_{B\cup C} f(y) \, \ud^2 y
	=:I_1+I_2+I_3.  
	\end{multline}
	Using \eqref{logass}, for $|x|\gg 2$ we estimate
	\begin{equation}\label{e-condition1}
	I_1=\frac{4}{|x|^2}\int_{|y|\le |x|}f(y)|y|\, \ud^2 y \le
	\frac{8
		\pi}{|x|^2}\underbrace{\int_0^{|x|}\varphi(t)t^2\,dt}_{o(|x|)} 
	=o(|x|^{-1})\qquad(|x|\to\infty).
	\end{equation}
	Also using the monotonicity of $f$ and
	\eqref{logass},
	for $|x|\gg 2$ we obtain
	\begin{equation}\label{e-condition2i}
	I_2=\int_{|y-x|\le |x|/2}\frac{f(y)}{|x-y|}\,dy\le
	\varphi(|x|/2)\int_{|z|\le |x|/2}\frac{dz}{|z|}= \pi
	\varphi(|x|/2)|x|=o(|x|^{-1}). 
	\end{equation}
	Finally, $I_3=o(|x|^{-1})$ as $|x|\to\infty$ since $f\in L^1(\R^2)$,
	so the assertion follows.
\end{proof}

\begin{proposition}\label{P-Z}
Assume that the assumptions of Proposition \ref{P-decay} holds and
\begin{equation}
  \lim_{|x|\to\infty} 2 \pi |x|V(x)=Z>0.
\end{equation}
Then $\|\rho_V\|_{L^1(\R^2)}=Z$.
\end{proposition}

\proof 
According to \eqref{eq:EL-weak+0}, the minimizer $\rho_V\in \mathcal H_0\cap C^{\alpha}(\R^2)$ satisfies
\begin{equation}
  \rho_V^{1/2}(x)=V(x)-U_{\rho_V}(x)\quad\text{for all $x\in\R^2$}.
\end{equation}
Taking into account \eqref{rho-decay}, by Lemma \ref{lemma-Newton}
above we conclude that
\begin{equation}
  \lim_{|x|\to\infty}2 \pi |x|U_{\rho_V}(x)=\|\rho_V\|_{L^1(\R^2)}.
\end{equation}
Then the assertion follows since $\lim_{|x|\to\infty}|x|\rho_V^{1/2}(x)=0$.
\qed

\subsection{Universality of decay}
\label{sec:uni}

We next prove that in the case $V=V_{Z,d}$ the behavior of
$\rho_{V_{Z,d}}$ for large $|x|$ does not depend on the values of $Z$
and $d$. 

\begin{proposition}\label{P-univers}
	Let $Z>0$, $d>0$ and let $V=V_{Z,d}$ as defined in \eqref{VZd}.
	Then 
	\begin{equation}\label{e-univers}
          u_V(x)\simeq\frac{1}{|x|\log|x|}\quad\text{as }|x|\to\infty.
	\end{equation}
\end{proposition}

\proof We start by noting that Proposition \ref{P-decay} applies to
$V_{Z,d}$, see \eqref{eq:VZdhalf}.  To prove the sharp asymptotic
decay of the minimizer when $V = V_{Z,d}$, we use the idea in the
computation of Katsnelson \cite{Katsnelson:06}, also giving the latter
a precise mathematical meaning. To this end, we first note that since
$\rho_V \in L^1(\R^2) \cap L^\infty(\R^2)$, we have that
\eqref{eq:EL-weak0} holds. In terms of $u_V > 0$ defined in
\eqref{e-Phiu0intr}, this equation reads
\begin{align}
	\label{eq:ELu0}
	u_V(x) = V(x) - {1 \over 2 \pi} \int_{\R^2} {u_V^2(y) \over
		|x - y|} \ud^2 y \qquad \text{for all } x \in \R^2,
\end{align}
where we used the regularity of $u_V$ and $V$.  In turn, since
$u_V(x) = u(|x|)$, applying Fubini's theorem we obtain after an
explicit integration:
\begin{align}
	\label{eq:uRiesz}
	u(r) & = {Z \over 2 \pi \sqrt{d^2 + r^2}} - {1 \over 2 \pi}
	\int_0^\infty \int_0^{2 \pi} {u^2(r')  \, r' \ud r' \ud \theta
		\over \sqrt{r^2 + {r'}^2 - 2 r r' \cos \theta}} \\
	& =   {Z \over 2 \pi \sqrt{d^2 + r^2}}  - {2 \over \pi}
	\int_0^\infty {r' u^2(r')
		\over r + r'} K \left( { 2 \sqrt{r r'} \over r + r'} \right) \ud r',
\end{align}
where $K(k)$ is the complete elliptic integral of the first kind \cite{Abramowitz}.

Proceeding as in \cite{Katsnelson:06}, we introduce a smooth bounded
function
\begin{align}
	\label{eq:Fu}
	F(t) := e^t u(e^t), \qquad t \in \R,
\end{align}
which satisfies $F(\ln r) = r u(r)$. From \eqref{eV-upperLog} and the
boundedness of $u$ we conclude that
\begin{align}
	\label{eq:Fut}
	F(t) \sim t^{-1}\quad\text{as $t \to+\infty$},
\end{align}
and $F(t)$ decays exponentially as $t\to-\infty$. In particular,
$F\in L^2(\R)$.  Then with the substitution $r = e^t$,
\eqref{eq:uRiesz} written in terms of $F(t)$ becomes
\begin{align}
	\label{eq:Ft}
	F(t) = {Z \over 2 \pi \sqrt{1 + d^2 e^{-2 t}}} - {2 \over \pi}
	\int_{-\infty}^\infty {F^2(t') \over 1 + e^{t' - t}} \, K \left( {1
		\over \cosh {t' - t \over 2} } \right) \ud t'. 
\end{align}
We further introduce (with the opposite sign convention to that in \cite{Katsnelson:06})
\begin{align}
	\label{eq:phi}
	\phi(t) :=  {2 K \left( {1 \over \cosh {t \over 2} } \right) \over
		\pi (1 + e^{-t})} - \theta(t),
\end{align}
where $\theta(t)$ is the Heaviside step function, and note that
$\phi(t)$ is a positive, exponentially decaying function as
$t \to \pm \infty$, which is smooth, except for a logarithmic
singularity at $t = 0$. Then, since $F(t) \to 0$ as
  $t \to +\infty$ by \eqref{eq:Fut}, passing to the limit using the
  weak convergence of $\phi(t-\cdot)\rightharpoonup 0$ in $L^2(\R)$ as
  $t\to+\infty$ and monotone convergence, we conclude that
\begin{align}
\int_{-\infty}^{+\infty} F^2(t) dt = \frac{Z}{2 \pi}.
\end{align}
With this \eqref{eq:Ft} becomes
\begin{align}
	\label{eq:Ft2}
	F(t) = {Z \left( 1 - \sqrt{1 + d^2 e^{-2 t}} \right) \over 2 \pi
		\sqrt{1 + d^2 e^{-2 t}}}  + \int_t^\infty F^2(t') \ud t' -
	\int_{-\infty}^{\infty} \phi(t - t') F^2(t') \ud t'. 
\end{align}

To conclude, we observe that in view of \eqref{eq:Fut} we can estimate
the last term in \eqref{eq:Ft2} to be $O(t^{-2})$ as $t \to
+\infty$. Similarly, the first term gives an exponentially small
contribution for $t \to +\infty$ and can, therefore, be absorbed into
the $O(t^{-2})$ term as well. Thus we have
\begin{align}
	\label{eq:Gt}
	F(t) = G(t) + O(t^{-2}), \qquad G(t) := \int_t^\infty F^2(t') dt',
\end{align}
and it follows that $G(t)$ satisfies for all $t$ sufficiently large
\begin{align}
	\label{eq:Gt2}
	{d G(t) \over dt} = -\left( G(t) + O(t^{-2}) \right)^2. 
\end{align}
In particular, using \eqref{eq:Fut} we can further estimate for
$t \gg 1$:
\begin{align}
	\label{eq:Gt3}
	{d G(t) \over dt} = -G^2(t) \left( 1 + O(t^{-1}) \right)^2. 
\end{align}
Integrating this expression from some sufficiently large $t_0$ then
gives
\begin{align}
	\label{eq:Gt4}
	{1 \over G(t)} - {1 \over G(t_0)} = t - t_0 + O(\ln (t / t_0)),
	\qquad t > t_0. 
\end{align}
Finally, solving for $G(t)$ and inserting it into \eqref{eq:Gt}
results in
\begin{align}
	\label{eq:Gt5}
	F(t) = {1 \over t + O(\ln t)} \qquad \text{as} \ t \to +\infty,
\end{align}
which yields the claim after converting back into the original variables.
\qed

  \subsection{Proof of the main results.}
    
  We finish this section by concluding the proofs of the results in
  section \ref{sec:results}.

  \begin{proof}[Proof of Theorem \ref{thm:Egeneral}]
    The statement of the theorem follows by combining the statements
    of Propositions \ref{p-Min0}, Lemma \ref{lem:regularity} and
    Proposition \ref{pPDEplus}.
  \end{proof}

  \begin{proof}[Proof of Theorem \ref{thm-main}]
    The conclusion of this theorem is the consequence of Theorem
    \ref{thm:Egeneral}, together with Corollary \ref{c-monotone} and
    Propositions \ref{P-decay}, \ref{P-Z} and \ref{P-univers}, taking
    into account \eqref{eq:VZdhalf} and performing a change of
    variables from $u_V$ to $\rho_V$.
  \end{proof}

 \begin{proof}[Proof of Corollary \ref{cor:cluster}]
   Now that we established Proposition \ref{P-univers} for the
     potential $V_{Z,d}$ with $Z > 0$ and $d > 0$, we may proceed to
     use the comparison principle in Lemma \ref{l-comparison0} to
     establish the sharp estimate in \eqref{e-univers} for a potential
     given by \eqref{eq:VN} with all $Z_i > 0$ and $d_i > 0$. In this
     case by \eqref{eq:VZdhalf} we have
  \begin{align}
    \label{eq:2}
    (-\Delta)^{1/2} V_N(x) = \sum_{i=1}^N {Z_i d_i \over  2 \pi (d_i^2 + |x -
    x_i|^2)^{3/2}}.
  \end{align}
  In particular, $V_N$ satisfies the assumptions of Proposition
  \ref{P-decay}. Hence the conclusions of Propositions \ref{P-decay}
  and \ref{P-Z} are still valid for $V = V_N$.

  It remains to establish the sharp decay estimate in
  \eqref{e-univers}. For that, simply observe that there exist
  constants $Z_2 > Z_1 > 0$ such that
  \begin{align}
    \label{eq:3}
    (-\Delta)^{1/2} V_{Z_1,1}(x) \leq  (-\Delta)^{1/2} V_N(x) \leq (-\Delta)^{1/2}
    V_{Z_2,1}(x) \qquad \forall x \in \R^2. 
  \end{align}
  Therefore, by Lemma \ref{l-comparison0} we have that
  $u_{V_{Z_1,1}} \leq u_{V_N} \leq u_{V_{Z_2,1}}$, and the conclusion
  follows from Proposition \ref{P-univers} and a change of variables
  from $u_{V_N}$ to $\rho_{V_N}$.
\end{proof}

\begin{remark}
  From the proof above it is clear that the universal decay estimate
  \eqref{e-univers} on the minimizer $\rho_V$ remains valid for any
  potential $V\in\HH$ such that $(-\Delta)^{1/2} V$ is nonnegative and
  bounded, and $(-\Delta)^{1/2} V\lesssim |x|^{-3}$ as $|x|\to\infty$.
\end{remark}

\section{Nonzero background charge}
\label{sec:setting}

We now turn to the situation in which a net background charge
density $\bar \rho\in\R$ is present, which is achieved
in graphene via back-gating. This leads to the modified TF-energy
\cite{LMM}
\begin{multline*}
  \E_\br^{TF}(\rho)=\frac{2}{3}\int_{\R^2}\big(|\rho
  (x)|^{3/2}-|\br|^{3/2}\big) \ud^2 x
  -\mathrm{sgn}(\br)|\br|^{1/2}\int_{\R^2}\big(\rho
  (x)-\br\big)\ud^2 x\\
  -\int_{\R^2}(\rho (x)-\br) V(x) \ud^2 x+\frac{1}{4
    \pi}\iint_{\R^2\times\R^2}\frac{(\rho(x)\!-\!\br)(\rho(y)\!-\!\br)}{|x-y|}\ud^2
  x \ud^2 y,
\end{multline*}
where $\rho(x)\to\br$ sufficiently fast as
$|x|\to\infty$.  Since this energy is invariant with respect to
$$\rho\to-\rho,\quad\br\to-\br,\quad V\to-V,$$
in the sequel we assume, without loss of generality, that $\bar \rho>0$.

\subsection{A representation of the energy functional} \label{sec:renormenergy}

For a given charge density $\rho(x)$ and $\br>0$, we
define
\begin{align}\label{u-rho}
\phi :=\rho - \bar \rho.
\end{align}
Then, for $\phi \in C^{\infty}_c(\R^2)$, the energy
$\E^{TF}_\br(\phi)$ can be written as (with a slight abuse of
notation, in what follows we use the same letter to denote both the
energy as a function of $\rho$ and that as a function of $\phi$)
\begin{equation}
\label{eq:E}
\E_\br^{TF}(\phi) =
\int_{\R^2} \Psi_\br(\phi(x)) \ud^2 x
- \int_{\R^2} V(x)\phi(x)\ud^2 x + \frac{1}{4\pi}
\iint_{\R^2 \times \R^2} \frac{\phi(x)\phi(y)}{|x - y|}  \ud^2 x \ud^2 y,
\end{equation}
where
\begin{align}\label{Phi}
  \Psi_\br(\phi):=\tfrac{2}{3}|\br+\phi|^{3/2}-\tfrac{2}{3}\br^{3/2}-\br^{1/2}\phi. 
\end{align}
Clearly $\Psi_\br:\R\to\R$ is a convex $C^1$--function of $\phi$ with
\begin{align}\label{Phi1}
	\Psi^\prime_{\br}(\phi)=|\br+\phi|^{1/2}\sgn(\br+\phi)-\br^{1/2},
\end{align}
and $\Psi_\br\in C^\infty(\R\setminus\{-\br\})$.  The graphs of
$\Psi_\br(\phi)$ and $\Psi^\prime_{\br}(\phi)$ for $\bar \rho = 1$ are
presented in Fig.~\ref{fig1}.

\begin{figure}[h]
	\centering
	\includegraphics[width=2.3in]{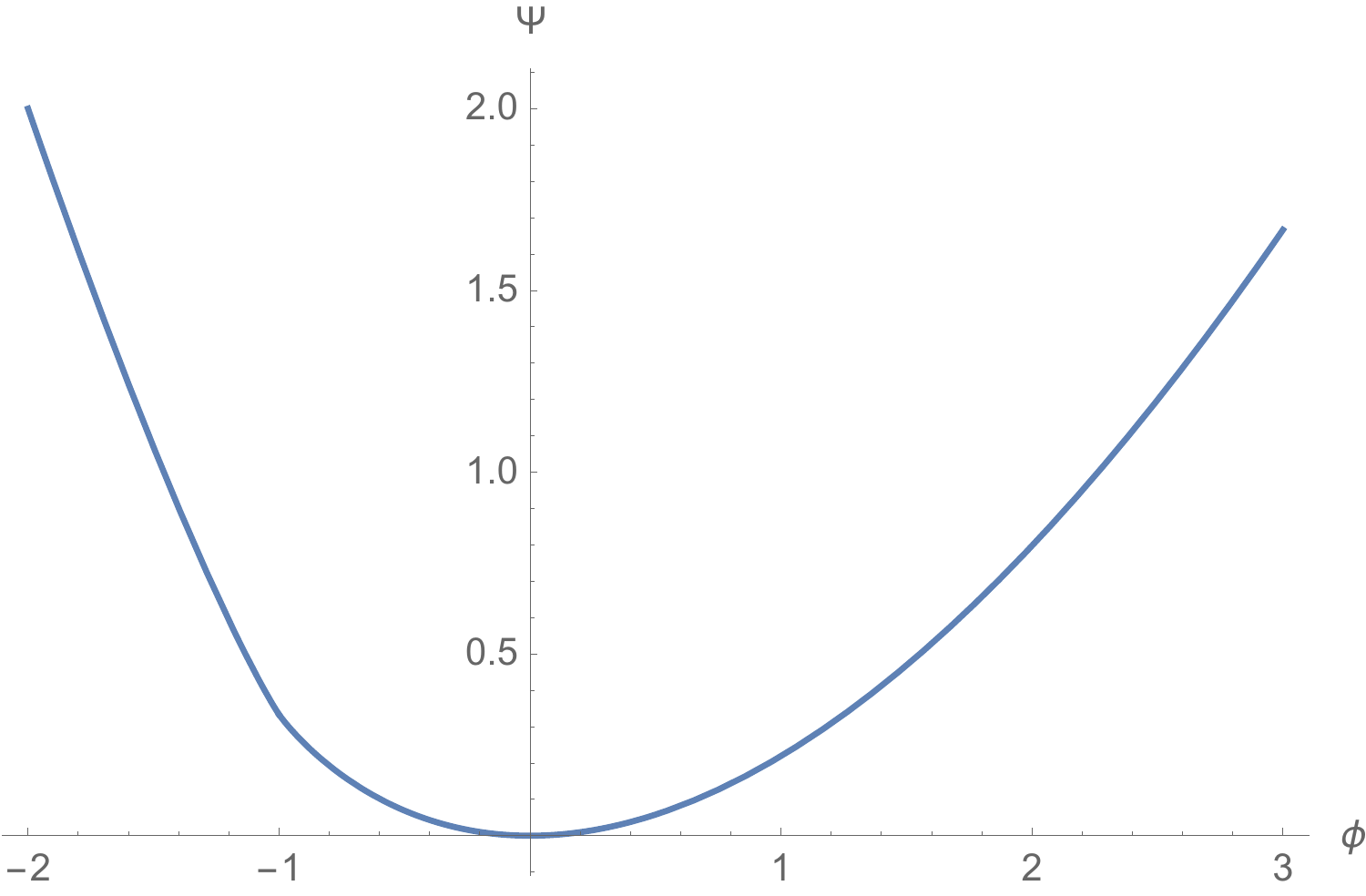}
	\qquad\qquad\qquad
	\includegraphics[width=2.3in]{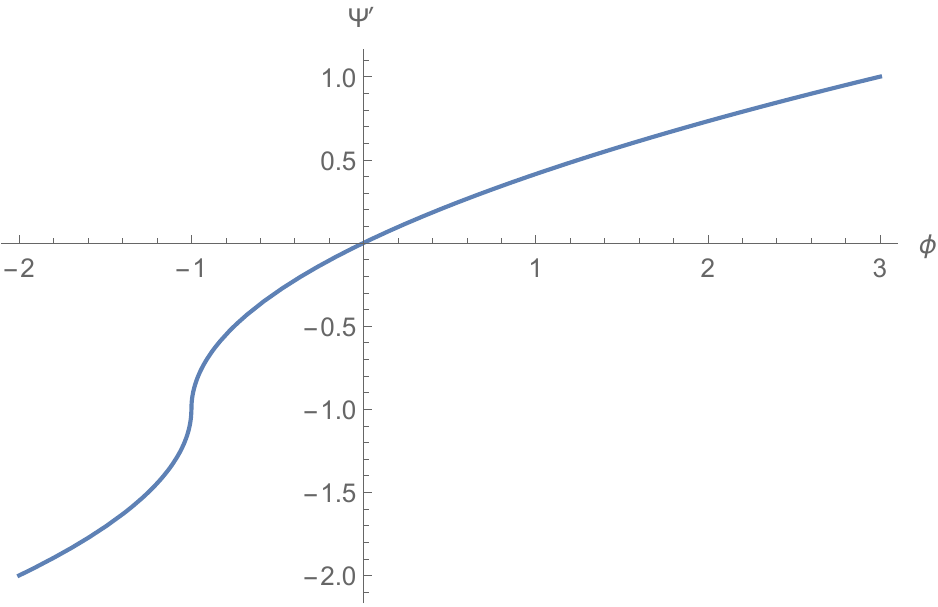}
	\caption{Plot of $\Psi_\br(\phi)$ and $\Psi^\prime_{\br}(\phi)$ for $\br = 1$.}
	\label{fig1}
\end{figure}

Using elementary calculus one can see that
\begin{align}
\label{Psi-abs}
\frac{c|\phi|^2}{\sqrt{\br+|\phi|}}\le\Psi_\br(\phi)\le
\frac{C|\phi|^2}{\sqrt{\br+|\phi|}}\qquad(\phi\in\R),
\end{align}
for some universal $C > c > 0$.  This implies that for $\br>0$,
\begin{align}\label{d-Psi}
  \Big\{\phi\in L^1_{\loc}(\R^2)\,:\,
  \|\Psi_\br(\phi)\|_{L^1(\R^2)}<+\infty\Big\}=L^{3/2}(\R^2)+L^2(\R^2). 
\end{align}

\begin{lemma}\label{l-lsc}
  Let $\br>0$. Then
  $\|\Psi_\br(\cdot)\|_{L^1(\R^2)}:L^{3/2}(\R^2) +L^2(\R^2)\to\R$
  is a strictly convex and weakly lower semi-continuous functional,
  i.e.
\begin{align}\label{e-lsc}
  \langle\phi_n,\varphi\rangle \to
  \langle\phi,\varphi\rangle\quad\forall\varphi\in L^3(\R^2) \cap
  L^2(\R^2)\quad\implies\quad 
  \|\Psi_\br(\phi)\|_{L^1(\R^2)}\le\liminf_n\|\Psi_\br(\phi_n)\|_{L^1(\R^2)}.
\end{align}
\end{lemma}

\proof
The strict convexity of $\|\Psi_\br(\cdot)\|_{L^1(\R^2)}$ follows from the strict convexity of the function $\Psi_\br:\R\to\R$.

Let $(\phi_n)\subset L^{3/2}(\R^2)+L^2(\R^2)$ be a sequence that
converges strongly to $\phi$, i.e. there exist representations
$\phi_n=f_n+g_n$ and $\phi=f+g$ such that
$\|f_n-f\|_{L^{3/2}(\R^2)}\to 0$ and $\|g_n-g\|_{L^2(\R^2)}\to 0$.
Then up to a subsequence $\Psi_\br(\phi_n)\to\Psi_\br(\phi)$
a.e. in $\R^2$. By Fatou's lemma,
\begin{align}
\|\Psi_\br(\phi)\|_{L^1(\R^2)}\le\liminf_n\|\Psi_\br(\phi_n)\|_{L^1(\R^2)},
\end{align}
i.e., the sublevel sets of $\|\Psi_\br(\cdot)\|_{L^1(\R^2)}$ are
closed in the norm of $L^{3/2}(\R^2)+L^2(\R^2)$.  Using the
convexity of $\|\Psi_\br(\phi)\|_{L^1(\R^2)}$, by Mazur's theorem we conclude that all sublevels sets are also weakly
closed in $L^{3/2}(\R^2)+L^2(\R^2)$, i.e. \eqref{e-lsc}
holds. 
\qed

\subsection{Variational setup and the main result.}
In view of Lemma \ref{l-lsc}, the natural domain of the total
TF--energy $\E_\br^{TF}$ is
\begin{align}
\label{eq:H}
\H_\br&:= \HHH\cap (L^{3/2}(\R^2)+L^2 (\R^2)),
\end{align}
and the TF-energy is correctly defined on $\H_\br$ in the form
\begin{align}
\label{eq:Eu}
  \E_\br(\phi) :=
  \int_{\R^2} \Psi_\br(\phi(x)) \ud^2 x
  - \langle \phi, V \rangle + \frac{1}{2}\|\phi\|_{\HHH}^2,
\end{align}
where $\langle \cdot, \cdot \rangle$ denotes the duality pairing
  between $\H_\br'$ and $\H_\br$.  Having in mind the definition of
$\H_\br$ in \eqref{eq:H}, we have
\begin{align}\label{V1}
  \H_\br^\prime& = \HH+(L^{3} (\R^2) \cap L^2 (\R^2)).
\end{align}
Our main result concerning minimizers of $\E_\br$ is the following.

\begin{theorem}\label{p-Min}
  Let $\br>0$ and $V\in\H_\br^\prime$.  Then $\E_\br$ admits a unique
  minimizer $\phi_{\br} \in \H_\br$ such that
  $\E_{\br}(\phi_{\br}) = \inf_{\H_\br}\E_\br$.  The minimizer
  $\phi_{\br}$ satisfies the Euler--Lagrange equation
  \begin{align}\label{e-EL}
    \int_{\R^2} \Psi^{\prime}_{\br}(\phi_{\br}(x))\varphi(x)\ud^2 x -
    \langle \varphi, V \rangle + \langle\phi_{\br},
    \varphi\rangle_\HHH=0\qquad\forall\varphi\in \H_\br. 
  \end{align}
\end{theorem}

\proof
The proofs of the existence and uniqueness of the minimizer (employing Lemma \ref{l-lsc}), as well as
the derivation of the Euler--Lagrange equations \eqref{e-EL} are small modifications of
the arguments in the proof of Proposition \ref{p-Min0}, so we omit the
details. For the differentiability of the map $\Psi_{\br}$ see
\cite{LMM}*{Lemma 6.2}.
\qed

\begin{remark}
	If, for instance, $\phi_{\br}\in \H_\br\cap L^{4/3}(\R^2)$ then
	\eqref{e-EL} can be interpreted pointwise as
	\begin{align}
		\label{eq:EL-weak}
		\Psi^\prime_{\br}(\phi_{\br}(x))+\frac{1}{2\pi}\int_{\R^2}\frac{\phi_{\br}(y)}{|x-y|}\ud^2
		y=V(x)\quad\text{a.e.~in $\R^2$}. 
	\end{align}
	However in general, the Euler--Lagrange equation for $\E_\br$ should
	be understood as
	\begin{align}
		\label{eq:EL-weak+}
		\Psi^\prime_{\br}(\phi_{\br})+ U_{\phi_{\br}}=V\quad\text{in
			$\mathcal D^\prime(\R^2)$}, 
	\end{align}
	where
	$\phi_\br\in \mathcal H_\br\subset L^2(\R^2)+L^{3/2}(\R^2)$ and
	$U_{\phi_{\br}}\in \HH$ is the potential of $\phi_{\br}$ defined
	via \eqref{weak-potential}. 
\end{remark}

In the rest of the section, under some additional assumptions on $V$
we will use the equivalent half--Laplacian representation of
\eqref{e-EL} to establish further regularity and decay properties of
the minimizer $\phi_\br$ when $\br>0$.  Our crucial observation is
that unlike in the case $\br=0$, for $\br>0$ the minimizer $\phi_\br$
has the same fast polynomial decay as the Green function of
$(-\Delta)^{1/2}+1$ in $\R^2$, for all reasonably fast decaying
potentials $V$.

\begin{theorem}
  \label{p-Minreg}
  Let $\br>0$, $V\in\HH$ and $\phi_\br$ be the minimizer of
  $\mathcal E_\br$ from Theorem \ref{p-Min}.
\begin{enumerate}
\item[$(i)$] If $(-\Delta)^{1/2}V\in L^\infty(\R^2)$ then
  $\phi_\br\in H^{1/2}(\R^2)\cap C^{1/2}(\R^2)$.
\item[$(ii)$] If additionally, $(-\Delta)^{1/2}V\geq 0$, $V \neq 0$, and for some 
  $C>0$ we have
  \begin{equation}
    (-\Delta)^{1/2}V\le \frac{C}{(1+|x|^2)^{3/2}}\quad\text{in $\R^2$},
  \end{equation}
  then $\phi_\br>0$ in $\R^2$ and
  \begin{equation}
    \phi_\br(x)\sim\frac{1}{|x|^{3}}\quad\text{as }|x|\to\infty.
  \end{equation}
  In particular, $\phi_\br\in L^1(\R^2)$.
\end{enumerate}
\end{theorem}

In the rest of this section we are going to sketch the proof of
Theorem \ref{p-Minreg}. We only emphasize the difference in the
asymptotic behaviour, other arguments that are similar to the case
$\br=0$ will be omitted.

\subsection{Half--Laplacian representation, regularity and decay}

Let $\br> 0$ and $\phi_{\br}\in\H_\br$ be the minimizer of
$\E_\br$. Introduce the substitution
	\begin{align}\label{e-Phiu}
          u_{\br} := \Psi^\prime_{\br}(\phi_{\br}).
	\end{align}
	Then \eqref{e-EL} transforms into
	\begin{align}\label{e-EL-0}
          \int_{\R^2} u_{\br}(x)\varphi(x)\ud^2 x -
          \langle \varphi, V \rangle + \langle
          U_{S_{\br}(u_{\br})},\varphi\rangle_{\HH}=0\qquad\forall\varphi\in
          \H_\br, 
	\end{align}
	where
	\begin{align}\label{e-PhiS}
          S_{\br}(u):=|\br^{1/2}+u|(\br^{1/2}+u)-\br\qquad(u\in\R) 
	\end{align}
	is the inverse function of $\Psi^\prime_{\br}$, so that
        $S_{\br}(\Psi^\prime_{\br}(\phi))=\phi$, for all
        $\phi\in\R$. The graph of $S_{\br}(u)$ is shown in
          Fig. \ref{fig2}.
	
\begin{figure}[h]
	\centering
	\includegraphics[width=3in]{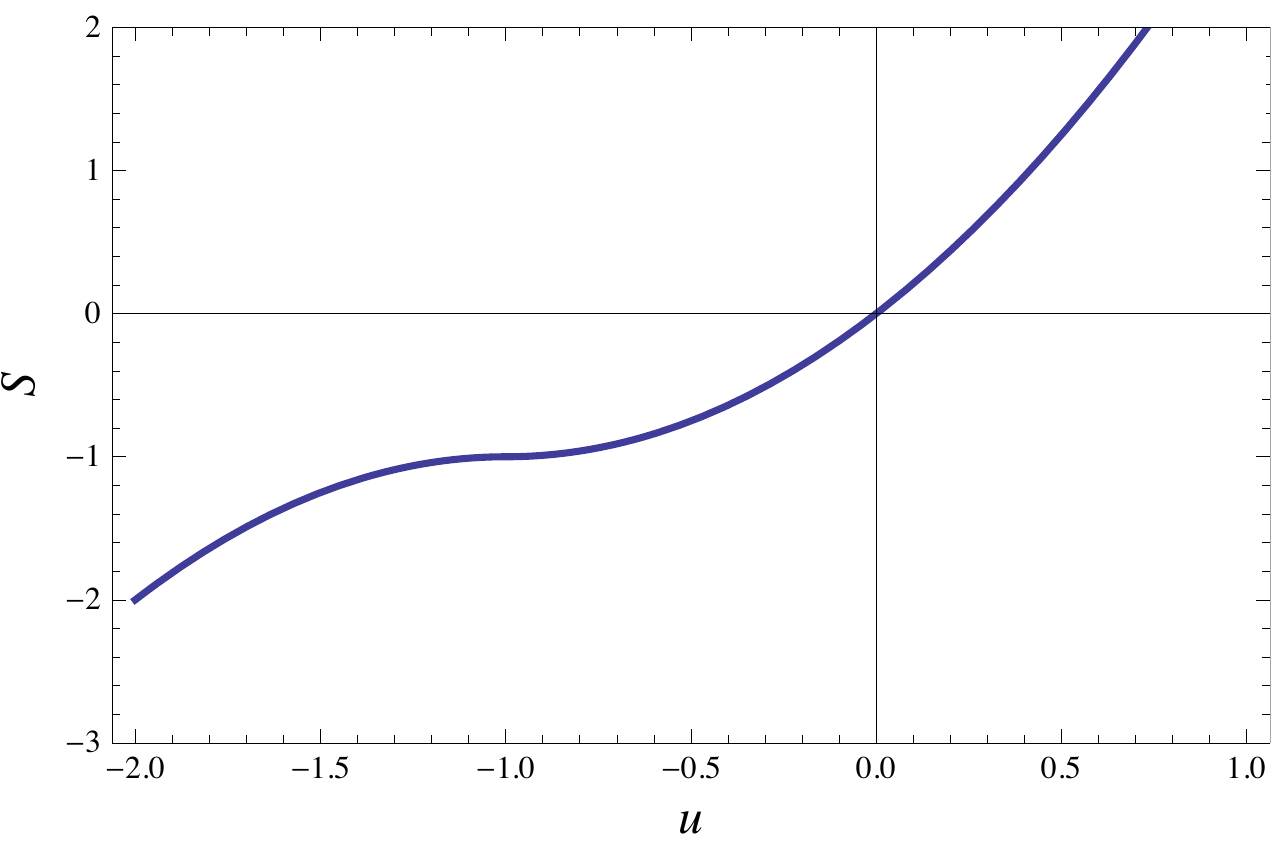}
	\caption{Plot of $S_{\br}(u)$ for $\br = 1$.}
	\label{fig2}
\end{figure}
	
\begin{proposition}[Equivalent PDE]\label{pPDE}
  Let $\br>0$, $V\in\HH$ and $u_\br$ be defined by
  \eqref{e-Phiu}. Then $u_\br\in\HH$ and $u_\br$ is the unique
  solution of the equation
  \begin{align}\label{ELu}
    (-\Delta)^{1/2} u + S_{\br}(u) = (-\Delta)^{1/2} V \quad  \text{in $\HH$}.
  \end{align}
  Moreover,
  \begin{align}\label{ELu-bounds}
    -v_-\le u_{\br}\le v_+,
  \end{align}
  where $v_\pm\ge 0$ are solutions of
  $(-\Delta)^{1/2} v_\pm=((-\Delta)^{1/2}V)^\pm$ in $\HH$.
\end{proposition}

\proof Similar to the proof of Propositions \ref{pPDE0} and \ref{pPDEplus}.
The uniqueness of the solution and the bound \eqref{ELu-bounds} follows from an extension of the comparison principle of Lemma \ref{l-comparison0} to the case of a monotone increasing function $S_{\br}(u)$.
\qed

\begin{proposition}\label{p-deacy-bar}
  Let $\br>0$ and $V\in\HH$. Assume that
  $(-\Delta)^{1/2}V\in L^\infty(\R^2)$, $(-\Delta)^{1/2}V\geq 0$ and
  $V\neq 0$. Then $u_\br\in H^{1/2}(\R^2)\cap C^{1/2}(\R^2)$,
  $u_\br>0$ in $\R^2$ and
  \begin{equation} u_\br(x)\gtrsim\frac{1}{|x|^3}\quad\text{as
    }|x|\to\infty.
	\end{equation}
	If, in addition, for some $C>0$,
	\begin{equation}
          (-\Delta)^{1/2}V\le \frac{C}{(1+|x|^2)^{3/2}}\quad\text{in $\R^2$},
	\end{equation}
 then
	\begin{equation}\label{cubic}
	u_\br(x)\sim\frac{1}{|x|^3}\quad\text{as }|x|\to\infty.
	\end{equation}
In particular, $u_\br\in L^1(\R^2)$.
	
\end{proposition}

\proof 
Represent \eqref{ELu} as
\begin{align}
\label{ELu++}
\big((-\Delta)^{1/2} + 2\br^{1/2}\big) u_{\br} + s_{\br}(u_{\br})
  = (-\Delta)^{1/2} V \quad  \text{in} \ \mathcal D'(\R^2),
\end{align}
where $s_{\br}(t)=S_{\br}(t)-2\br^{1/2} t$ and observe that
$s_{\br}(t)=t^2$ for $|t|<\br^{1/2}$ small, while
$s_{\br}(t)\sim|t|t$ for $t$ large.
In particular, in view of \eqref{ELu-bounds} we have $u_{\br}\ge 0$
and $u_{\br}\in L^\infty(\R^2)$. Then for a sufficiently large $c>0$,
\begin{align}
\label{ELu+++}
\big((-\Delta)^{1/2} + 2\br^{1/2}+c\big) u_{\br}=
  c-s_{\br}(u_{\br})+ (-\Delta)^{1/2} V\ge 0 \quad  \text{in $\R^2$}. 
\end{align} 
This implies $u_{\br}\in H^{1/2}(\R^2)\cap C^{1/2}(\R^2)$,  $u_\br>0$ in $\R^2$ and  additionally,
\begin{equation}
u_\br(x)\gtrsim\frac{1}{|x|^3}\quad\text{as }|x|\to\infty,
\end{equation}
cf. \cite{LMM}*{Lemma 7.1} for a similar argument. 

To derive the upper bound on $u_\br$, consider the dipole type family of barriers
 $$W_{Z,\lambda}(|x|):=\frac{Z}{2\pi(1+|\lambda x|^2)^{3/2}}$$ and note that using \eqref{e-dipole}, scaling, $s_\br(W_{Z,\lambda})\ge 0$ and \eqref{cubic}, we obtain
\begin{multline}
	\label{ELu+++up}
	\big((-\Delta)^{1/2} + 2\br^{1/2}\big) W_{Z,\lambda} +
        s_{\br}(W_{Z,\lambda})- (-\Delta)^{1/2} V
	\\
	\ge\frac{Z\lambda(2-|\lambda x|^2)}{2\pi(1+|\lambda
          x|^2)^{5/2}}+\frac{2 Z\br^{1/2}(1+|\lambda
          x|^2)}{2\pi(1+|\lambda
          x|^2)^{5/2}}-\frac{C}{(1+|x|^2)^{3/2}} \ge 0 \quad \text{in
          $\R^2$} ,
\end{multline} 
provided that we choose $\lambda=2\br^{1/2}$ and $Z\gg 1$ sufficiently large.
Then $u_\br\le W_{Z,2\br^{1/2}}$ in $\R^2$ by an extension of the comparison principle of Lemma \ref{l-comparison0} to the equation \eqref{ELu++}.
\qed

\begin{proof}[Proof of Theorem \ref{p-Minreg}]
Follows from Proposition
\ref{p-deacy-bar} using the explicit representation
$\phi_\br=S_\br(u_\br)=2\br^{1/2} u_\br + u_\br^2$ in \eqref{e-PhiS},
which is valid since $u_\br>0$.
\end{proof}


\paragraph{\bf Data availability statement.} Data sharing not applicable
to this article as no datasets were generated or analysed during the
current study.

\bibliographystyle{plain}
\bibliography{TF-graphene-biblio}

\end{document}